\documentclass[3p]{elsarticle}
\usepackage{amssymb}
 \usepackage{amsthm}
\usepackage{amsmath,bbm,dsfont, color}
\usepackage{srcltx}

\usepackage[OT2, T1]{fontenc}

\usepackage[cp1250]{inputenc}

\journal{arXiv}

\begin{document}

\newcommand*{\id}{\mathds{1}}
\newcommand*{\zahlen}{\mathbb{Z}}
\newcommand*{\en}{\mathbb{N}}
\newcommand*{\er}{\mathbb{R}}
\newcommand*{\FF}{\mathcal{F}}
\newcommand{\divv}{\mathop{\text{div}\,}}
\newcommand{\esssup}{\mathop{\text{ess\:sup}\,}}
\newcommand{\htimes}{\mathop{\text{\large$ďż˝$}}}
\newcommand{\hexists}{\mathop{\text{\LARGE$\exists$}}}
\newcommand{\hforall}{\mathop{\text{\LARGE$\forall$}}}
\newtheorem{lem}{Lemma}
\newtheorem{theo}{Theorem}
\newtheorem{cor}{Corollary}
\newtheorem{prop}{Proposition}
\newtheorem{mydef}{Definition}
\newtheorem{rem}{Remark}
\newtheorem{ass}{Assumption}

\begin{frontmatter}

 \title{Interior regularity of space derivatives to an evolutionary, symmetric $\varphi$-Laplacian}

\author{Jan Burczak} \address{Institute of Mathematics, Polish Academy of Sciences, \'Sniadeckich 8, 00-656 Warsaw, Poland.}

\author{Petr Kaplick\'y }  \address{Department of Mathematical Analysis, Faculty of Mathematics and Physics,
Charles University in Prague,
186 75 Praha 8,
Czech Republic
}



\begin{abstract}
We consider Orlicz-growth generalization to evolutionary $p$-Laplacian and to the evolutionary symmetric $p$-Laplacian. We derive the spatial second-order Caccioppoli estimate for a local weak solution to these systems. The result is new even for the $p$-case.
\end{abstract}





\end{frontmatter}






\newenvironment{rcases}
  {\left.\begin{aligned}}
  {\end{aligned}\right\rbrace}


%

%


\newcommand{\ep}{\varepsilon}
\newcommand{\eps} {\varepsilon}


\newcommand{\Div}       {{\rm div}_x}
\newcommand{\toc}       {{\stackrel{b}{\longrightarrow\,}}}

\newcommand{\A}     {{\mathcal A}}

\newcommand{\V}     {{\mathcal V}}
\newcommand{\DD}     {{\mathcal D}}

\newcommand{\D}     {{\rm D}}

\def\tens#1{\pmb{\mathsf{#1}}}
\def\vec#1{\boldsymbol{#1}}

\newcommand{\eN}	{{\mathcal{N}}} 
\newcommand{\M}	{{\mathcal{M}}} 
\newcommand{\eO}	{{\mathcal{O}}} 
\newcommand{\T}	{{\mathcal{T}}}
\newcommand{\W}	{\mathbb W} 
\newcommand{\ti}	{\tilde}
\newcommand{\wti}	{\widetilde} 
\newcommand{\vep}{\varepsilon}

\newcommand{\vdd}{\tens{\bar{v}}} 
\newcommand{\Ref} {\eqref}

\newcommand{\comentario}[1]{\textcolor{red}{\sc[*** #1 ***]}}





%

\pagestyle{headings} 

%


%

%

\newcommand{\Sym}{{\rm Sym}}

\newcommand{\vfi}{{\varphi}}
\newcommand{\Vfi}       {\bar \varphi}

 \def\Xint#1{\mathchoice
 {\XXint\displaystyle\textstyle{#1}}%
 {\XXint\textstyle\scriptstyle{#1}}%
 {\XXint\scriptstyle\scriptscriptstyle{#1}}%
 {\XXint\scriptscriptstyle\scriptscriptstyle{#1}}%
 \!\int}
 \def\XXint#1#2#3{{\setbox0=\hbox{$#1{#2#3}{\int}$}
 \vcenter{\hbox{$#2#3$}}\kern-.5\wd0}}
 \def\ddashint{\Xint=}
 \def\dashint{\Xint-}

\def\B{\mathbb{B}}
\def\test{\mathcal{D}}
\def\F{\mathbb{F}}
\def\N{\mathbb{N}}
\def\O{\Omega}
\def\R{\mathbb{R}}
\def\Kor{{\rm Kor}}
\def\pod#1{\mathop{#1}\limits}
\def\diagin{-\hskip-11.0truept\intop}
\def\diagint{{\raise-.1pt\hbox{--}\hskip-7.9pt\intop}}
\def\diagintop{\mathop{\mathchoice
{{\diagin}}%
{{\diagint}}%
{{\diagint}}%
{{\diagint}}%
}\limits}
 \def\ddashint{\Xint=}
 \def\dashint{\Xint-}

%
%
%
%

\section{Introduction} We derive interior (local) regularity of space derivatives to
\begin{equation}\label{eq:jb_structure:weak}
u,_{t}-\,\divv \A(\D u)=f, \tag{s$\A$L}
\end{equation}
where we have the symmetric gradient $\D u = \frac{1}{2} (\nabla u + \nabla^T u)$,
and to
\begin{equation}\label{eq:jb_structure:weakX}
u,_{t}-\,\divv \A(\nabla u)=f \tag{$\A$L}
\end{equation}
for tensors $\A$ having certain Orlicz-type growths that generalize $p$-growths. 
Our results are also new for the evolutionary $p$-case and our use of Orlicz growths is motivated primarily by elegance of this approach and by our need to derive results useful for planned interior $C^{1, \alpha}$ regularity in two-dimensional case. Let us explain the latter point. In order to use the method of Kaplick\'y, M\'alek, Star\'a \cite{KapMalSta02} for the intended $C^{1, \alpha}$ result, we need to derive regularity estimates for quadratic approximations of $p$-potentials uniformly in the approximation parameter, which is exactly where Orlicz-growth characteristics prove to be helpful. Beyond the pure analytic interest, the local estimates are important as building blocks for the non-linear Calder\'on-Zygmund theory. 
 
 In this paper we obtain local regularity results connected with the second-order energy estimates in space (i.e. testing with the localized Laplacian, roughly speaking).

\subsection*{Motivation and known results}
It would be crucial for the non-Newtonian hydrodynamics to repeat for the symmetric  $p$-Laplacian the $C^{1, \alpha}_{loc}$ regularity result, available for the full gradient $p$-Laplacian, since the works of Uhlenbeck \cite{Uhl77}, Tolksdorff \cite{Tol83} (stationary case) and DiBenedetto \cite{DiB93} (evolutionary case). Unfortunately, the pointwise structure of the symmetric  $p$-Laplacian seems to be resistant to  the methods used in the full $p$-Laplacian case\footnote{For more on this, see \cite{Burphd} and \cite{Bur14}} to get boundedness of gradients. 

Nevertheless, one can obtain certain regularity results for the symmetric  $p$-Laplacian. For the stationary case, Beir\~ao  da Veiga \cite{BdV} and  Beir\~ao  da Veiga \& Crispo \cite{daVCri13dcds} provide smoothness of a periodic-boundary value problem, provided $p$ is close to $2$. Results for generic boundary-value problems, developed for the full $p$-Navier-Stokes system, are of course available for the symmetric $p$-Laplacian. In particular, one has smoothness of solutions to basic initial-boundary value problems in $2$d case, see Kaplick\'y, M\'alek, Star\'a \cite{KapMalSta02} and  Kaplick\'y \cite{Kap05, Kap08} as well as existence of strong solutions for $3$d case, compare \cite{MalNecRuz01} by M\'alek, Ne\v cas and R\r u\v zi\v cka and also \cite{Vei09a, Vei09b} by Beir\~ao  da Veiga and \cite{VeiKapRuz11} by Beir\~ao  da Veiga,  Kaplick\'y and R\r u\v zi\v cka. Let us mention also here a recent regularity study for the symmetric $p$-Laplacian with homogenous Dirichlet boundary conditions by Frehse and Schwarzacher \cite{FreSchARXIV}. For small data regularity results, one may refer to Crispo \& Grisanti \cite{CriGri08}. 

As remarked, the regularity results cited above concern certain basic boundary-value problems. Local (interior) regularity results are much more scarce. In \cite{Bur14} the partial $C^{1, \alpha}_{loc}$ regularity theory has been developed. One should mention also \cite{FucSer} by Fuchs and Seregin 

On the other hand, the regularity results for the full-gradient case are abundant. For the evolutionary  $p$-case, let us restrict ourselves to referring to the classical monograph by DiBenedetto \cite{DiB93} and a simple proof of $C^{1, \alpha}_{loc}$ regularity by Gianazza, Surnachev \& Vespri \cite{GiaSurVes10}. For the stationary Orlicz case, see \cite{[DSV]} by Diening, Stroffollini \& Verde. 

\subsection*{Outline of the paper}
Next section involves needed preliminaries, including a short discussion of the used Orlicz-type spaces and notion of a weak solution. In the subsequent section we state our main results. The successive Section \ref{sec:aux} is devoted to certain auxiliary results. Section \ref{sec:prf} contains proof of our main results. Finally, we gather in the last Section \ref{jb_structure_orlicz} -- Appendix details of the used Orlicz growths. In its last subsection we recall the standard examples for Orlicz growths that are admissible in our main results.

\section{Preliminaries}

By $a \sim b$ we mean that $a$ and $b$ are equivalent up to a numerical constant, {\em i.e.} there exists $C> 0$ that $C^{-1} a \le b \le C a$. By $ Sym^{d \times d}$ we denote the set of symmetric ${d \times d}$ matrices. $\Omega$ is a bounded (space) domain in $\er^d$, since we deal with local solutions, any further description of $\Omega$ is immaterial. $B_r$ is a (space) ball of radius $r$. $I=[a, b]$ is a (time) interval. By  $b$ we will denote an interval of length $2r$. By $\Omega_I$  we mean $\Omega \times I$. A parabolic cylinder $Q_\rho$ is  $B_\rho \times I_{\rho^2}$. We introduce further notation when it is needed.
\subsection*{Orlicz growths}
For clarity, since we are still before stating our main results, let us introduce here only the essential definitions and assumptions. A more precise discussion of them can be found in Section \ref{jb_structure_orlicz} -- Appendix.

We use the standard definition of a $\eN$-function $\vfi$ and its conjugate  $\vfi^*$, see Subsection \ref{ssec:jb_structure_orlicz:N}.
Let us introduce
\begin{mydef}[good $\vfi'$ property] \label{def:jb_weak:goodphi'}
$\eN$-function $\vfi \in C^{2} ( (0, \infty) ) \cap C^1 ( [0, \infty))$ has { \emph{good $\vfi'$ property}} iff uniformly in $t$
\begin{equation}\label{eq:jb_weak:goodphi'}
 \varphi' (t)  \sim t \varphi'' (t).
\end{equation}
We denote the optimal constants in $\sim$ above by $G (\vfi')$.
\end{mydef}
An analogous class of $\eN$-functions  is considered in Diening \& Ettwein \cite{DieEtt08} or Diening \& Kaplick\'y \cite{DieKap13}, see Assumption 1 and Assumption 1.1 therein, respectively. Let us remark immediately that  \emph{good $\vfi'$ property} implies that he complementary function $\vfi^*$ also enjoys good $(\vfi^*)'$ property and that both $\vfi$ and $\vfi^*$ satisfy also $\Delta_2$ condition, see Proposition \ref{prop:jb_structure:'giveD}. From now on, we will work only with $\eN$-functions with \emph{good $\vfi'$ property}.

The connection between a $\eN$-function $\vfi $ and the tensor $\A$ is given by
\begin{ass}[Orlicz growth of  $\A$]\label{ass:jb_weak:growth_gen}
For an $\eN$-function $\vfi$ that has the good $\vfi'$ property, the tensor $\A$ satisfies for any $P, Q \in Sym^{d \times d}$
\begin{equation*}
\begin{aligned}
(\A(P) - \A(Q) ) \! : \! (P- Q) & \ge  c \vfi'' (|P| + |Q|) |P-Q|^2 \\
| \A(P) - \A(Q)|& \le C \vfi'' (|P| + |Q|) |P-Q|
\end{aligned}
\end{equation*}
with numerical constants $c, \, C$.
\end{ass}
The above assumption typically generalizes monotonicity and growth of tensor $\A$ from the polynomial to the Orlicz setting in the context of regularity theory, compare Diening \& Ettwein  \cite{DieEtt08}, Diening, Kaplick\'y \& Schwarzacher \cite{DieKapSch12}.

\subsection*{Orlicz-Lebesgue spaces}
In order to define  a weak solution, we will use the Orlicz-Lebesgue class $\Lambda^\vfi (\O)$. It includes all measurable functions $f\!: \, \O \to \er$ such that
\[
\int_\O \vfi (|f|) < \infty
\]
The Orlicz-Lebesgue space $L^\vfi (\O)$ consists of these measurable functions defined a.e. on $\O$ for which 
\begin{equation*}
|u|_\vfi := \sup_{\{ v \in \Lambda^{\vfi^*} (\O) | \,  \int_\O \vfi^* (|v|) \le 1 \}} \int_\O | u v|
\end{equation*}
is finite. Space  $L^\vfi (\O)$ is a Banach space with norm $| \cdot |_\vfi$ --- see Theorem 3.6.4 in \cite{KufJohFuc77}. One can define equivalent norm in $L^\vfi $ space without use of a complementary function, which is commonly referred to as the Luxemburg norm. 

Recall that if $\vfi$ satisfies $\Delta_2$ condition then $L^\vfi (\O)$ is separable and $\test (\O)$ is a dense set (in the norm topology). If also $\vfi^*$ satisfies $\Delta_2$ condition, then $L^\vfi (\O)$ is reflexive. Since we work here with growths that have good $\vfi'$ property, all these properties are valid for  $L^\vfi (\O)$ in our case via Proposition \ref{prop:jb_structure:'giveD} in Appendix.

\subsection*{Orlicz-Sobolev spaces}

For brevity, we will here use $S$ to denote either a domain being a subset of $\O$ or of $\O_I$, depending on the context. For all our domains we assume the segment property.
Let us define the Orlicz-Sobolev space as
\[
W^{1,\vfi}_x (S) = \{ u \in L^\vfi (S): \nabla u \in L^\vfi (S) \},
\]
the subscript $x$ above emphasizes that the space gradient $\nabla$ may be not the full gradient on $S$, which is the case for $S = \O_I$.
$W^{1,\vfi}_x (S)$ is a Banach space with the norm induced by the Orlicz-Lebesgue space $L^\vfi$, see Elmahi \& Meskine \cite{ElmMes05}, Section 2.5\footnote{Our $W^{1,\vfi}_x (Q)$ is $W^{1,x} L_\vfi (Q) = W^{1,x} E_\vfi (Q)$ of \cite{ElmMes05}, where the equality is valid since we work within $\Delta_2$ condition, see \cite{ElmMes05}, Section 2.5. Analogously, our $W^{1,\vfi}_x (\O)$ is theirs $W^{1} L_\vfi (\O) = W^{1} E_\vfi (\O)$.}.
We denote by $W^{1,\vfi}_{x, 0} (S)$ the norm closure of  $\DD (S)$ in $W^{1,\vfi}_x (S)$\footnote{Here, again, all topological ambiguities are disposed of by our assumptions. In particular, our $W^{1,\vfi}_{x,0} (S)$ coincides with  $W_0^{1,x} L_\vfi (S) = W_0^{1,x} E_\vfi (S)$ of \cite{ElmMes05}. The former is defined there as the weak closure and the latter as a norm closure of $\test (S)$. But for $S$ with the segment property the weak closure of  $\test (S)$ coincides with the closure in modular (see  \cite{ElmMes05}, Section 2.4 for $S = \Omega$ and Section 2.6 for $S = \Omega_I$), whereas the modular and norm convergence are equivalent for $\Delta_2$ growths.}. The dual space to  $W^{1,\vfi}_{x, 0} (S)$ is  denoted by $W^{-1,\vfi^*}_{x} (S)$ and its elements are representable with the distributional $L^2$-duality pairing
\[
\langle  f, \vfi \rangle_{W^{-1,\vfi^*},  W^{1,\vfi}_{0}}  = \int_Q f^\alpha \nabla^\alpha \vfi,
\] compare \cite{ElmMes05}, Sections 2.4, 2.6.

Finally, we denote by $W^{-1,\vfi^*}_{x} (Q) + L^2 (Q)$ the dual space to $W^{1,\vfi}_{x,0} (Q) \cap L^2 (Q)$. Its arbitrary element $f$ is representable as $ f = \sum_{ \,|\alpha| \le 1} f^\alpha + f^0$, where $ f^\alpha \in L^\vfi (Q)$ and $f^0 \in L^2 (Q)$,   under the distributional $L^2$-duality pairing, i.e. for any $\vfi \in W^{1,\vfi}_{x,0} (Q) \cap L^2 (Q)$ 
\begin{equation}\label{eq:modular}
\langle  f, \vfi \rangle_{W^{-1,\vfi^*} + L^2,  W^{1,\vfi}_{0} \cap L^2}  = \int_Q f^\alpha \nabla^\alpha \vfi + f_0 \vfi.
\end{equation}

\subsection*{Notion of a local weak solution}

We will use the following  notion of a (space-local) weak solution to  \eqref{eq:jb_structure:weak}.

\begin{mydef}\label{def:weak}
Take an interval $I=[a, b]$, a domain $\O \subset \er^d$ and a stress tensor $\A$ compatible with Assumption \ref{ass:jb_weak:growth_gen}. Let $f \in L^2 (\O_I)$. A function $u\in C(I;L^2(\Omega))$ with $\nabla u\in L^\varphi(\Omega_I)$ is called a local weak solution to the problem \eqref{eq:jb_structure:weak} iff  for any $t_1, t_2$ such that $(t_1, t_2) \subset I$  holds 
\begin{equation}\label{eq:jb_weak:e_ws}
\int_\O u (t ) \cdot w (t){ \Big|}^{t_2}_{t_1} + \int^{t_2}_{t_1} \int_\O -\, u \cdot w,_t + \, \A (\D u)  \D w =   \int^{t_2}_{t_1} \int_\O f w
\end{equation} 
for an arbitrary test function $w \; \in \; W^{1,1} (I; L^2 (\O)) \cap W^{1,\vfi}_{x,0} (\O_I) $.
\end{mydef}
Finiteness of the main part integral in \eqref{eq:jb_weak:e_ws} is assured by Assumption \ref{ass:jb_weak:growth_gen}, see Section \ref{jb_structure_orlicz}. The local weak solution for the full-gradient problem \eqref{eq:jb_structure:weakX} is defined analogously. 
\begin{rem}\label{rem1}
Observe that  $\O_I$ above may be smaller than the domain of existence of a weak solution to an initial-boundary value problem associated with the problem \eqref{eq:jb_structure:weak}.  The requirement that  $w$ has zero space-trace on $\partial \O$ gives us space-locality. Furthermore, in our main Theorem \ref{th1} we are interested in fact in space-time localization, hence there one may restrict to considering test functions vanishing on $\partial(\O_I)$. We allow for $w$ non vanishing on the  time-interval ends for the sake of  being able to choose $t_1 = a$ in \eqref{eq:jb_strong:2+_nabla'X}.
\end{rem}

\begin{rem}
 One can impose for $u$ merely $L^\infty_{loc}(I;L^2_{loc}(\Omega))$ regularity instead of the above  $C_{loc}(I;L^2_{loc}(\Omega))$ and then obtain continuity in time via the interpolation, see Lemma 3, p. 418 in \cite{ElmMes05}. 
\end{rem}

\subsection*{Existence of weak solutions to initial-boundary value problems}
For the existence of weak solutions to problem with full space gradient, we refer for instance to Elmahi \& Meskine \cite{ElmMes05}, Theorem 2\footnote{Their result concerns zero-Dirichlet boundary data and $f \in W^{-1,\vfi^*}_{x}$, but we will need more regularity of $f$ for our regularity results anyway.}. It gives additionally that $u_t$ is in the appropriate dual space and the energy equality. However, we decided to keep  less restrictive  notion of a weak solution, since it complies with existence results for very more general Orlicz and Musielak-Orlicz growths, in particular allowing for some anisotropy, compare \cite{SwG2014} by \'Swierczewska-Gwiazda and its references. These existence results can be straightforwardly rewritten for the symmetric-gradient case.


\section{Main results} Recall that a parabolic cylinder $Q_\rho$ is  $B_\rho \times I_{\rho^2}$. The $\V$ below is the \emph{square root tensor} given by Definition \ref{def:jb_weak:V}. We call a real function $g$ almost increasing iff there exists a number $C$ that for any $x \le y$ 
\[
g (x) \le C g(y).
\]
The result concerning \eqref{eq:jb_structure:weak} assumes additionally that $\vfi''$ is  almost increasing. For the full-gradient case this additional assumption is not necessary.

\subsection*{Result for the symmetric-gradient case  \eqref{eq:jb_structure:weak}}
\begin{theo}[spatial strong solutions for \eqref{eq:jb_structure:weak}]\label{lem:jb_strong:nabla_2+}\label{th1}
 If $\A$ satisfies Assumption  \ref{ass:jb_weak:growth_gen} with $\vfi''$ almost increasing, then a local weak solution $u$ to \eqref{eq:jb_structure:weak} on $\O_I$ enjoys
\[
\nabla u \in L_{loc}^\infty (I; L_{loc}^2 (\O)), \quad  \nabla \V (\D u) \in L_{loc}^2 (\O_I)
\]
with the following estimate 
  \begin{equation}\label{eq:jb_strong:2+_nabla}
\esssup_{\tau \in I_{r^2}} \int_{B_r} | \nabla u |^2 (\tau)+  \int_{Q_r}  | \nabla  \V( \D u) |^2 \le \frac{C ( G(\vfi'))(1+ \frac{1}{\vfi'' (\delta_0)})}{({R}-r)^2}   \int_{Q_{R}}   ( \vfi (|\nabla u|) + \vfi (\delta_0)  )+ C (R-r)^2 \int_{Q_R} | \nabla f|^2 
\end{equation}
for any $r < R$ and concentric parabolic cylinders $Q_r, Q_R \Subset \O_I$ and any $\delta_0 \ge 0$ such that $\vfi'' (\delta_0) > 0$ (there is always such $\delta_0$ available). 

We can also have a version of \eqref{eq:jb_strong:2+_nabla} without full gradient on the r.h.s., namely
  \begin{multline}\label{eq:jb_strong:2+_nablaK}
\esssup_{\tau \in I_{r^2}} \int_{B_r} | \nabla u |^2 (\tau)+  \int_{Q_r}  | \nabla  \V( \D u) |^2 \le \\
 \frac{C ( G(\vfi'))(1+ \frac{1}{\vfi'' (\delta_0)})}{({R}-r)^2}   \int_{Q_{R}}   ( \vfi (|\D u|) +  \vfi \left( \frac{|u - (u)|}{R} \right) + \vfi (\delta_0)  )+ C (R-r)^2 \int_{Q_R} | \nabla f|^2.
 \end{multline}
Moreover, in the case when $\vfi'' (0) > 0$, we have also
\[  \nabla^2 u \in L_{loc}^2 (\O_I) \]
and  can improve \eqref{eq:jb_strong:2+_nabla} and \eqref{eq:jb_strong:2+_nablaK}  to, respectively
  \begin{multline}\label{eq:jb_strong:2+_nabla'}
\esssup_{\tau \in I_{r^2}} \int_{B_r} | \nabla u |^2 (\tau)+  \int_{Q_r}  | \nabla  \V( \D u) |^2  + \vfi'' (0) \int_{Q_r}  | \nabla^2 u |^2 \le \\
  \left(1 + \frac{1}{\vfi'' (0)} \right)   \frac{C ( G(\vfi'))}{({R}-r)^2}   \int_{Q_{R}}    \vfi (|\nabla u|) + C (R-r)^2 \int_{Q_R} | \nabla f|^2  .
\end{multline}
  \begin{multline}\label{eq:jb_strong:2+_nabla'K}
\esssup_{\tau \in I_{r^2}} \int_{B_r} | \nabla u |^2 (\tau)+  \int_{Q_r}  | \nabla  \V( \D u) |^2  + \vfi'' (0) \int_{Q_r}  | \nabla^2 u |^2 \le \\
  \left(1 + \frac{1}{\vfi'' (0)} \right)   \frac{C ( G(\vfi'))}{({R}-r)^2}   \int_{Q_{R}}  \left(  \vfi (|\D u|) +  \vfi \left( \frac{|u - (u)|}{R} \right)  \right)+ C (R-r)^2 \int_{Q_R} | \nabla f|^2  .
\end{multline}
\end{theo}
\subsection*{Result for the full-gradient case  \eqref{eq:jb_structure:weak}}
As already remarked, one can do better in  the full-gradient case and drop the assumption of almost-increasingness of $\vfi''$ in the full-gradient case  \eqref{eq:jb_structure:weakX}. Namely, we obtain
\begin{theo}[spatial strong solutions  for \eqref{eq:jb_structure:weakX}]\label{lem:jb_strong:nabla_2+}\label{th2}
 If $\A$ satisfies Assumption  \ref{ass:jb_weak:growth_gen}, then a local weak solution $u$ to \eqref{eq:jb_structure:weakX} on $\O_I$ satisfies the following estimate 
  \begin{equation}\label{eq:jb_strong:2+_nablaX}
\esssup_{\tau \in I_{r^2}} \int_{B_r} | \nabla u |^2 (\tau)+  \int_{Q_r}  | \nabla  \V( \nabla u) |^2 \le \frac{C ( G(\vfi')}{({R}-r)^2}   \int_{Q_{R}} (   \vfi (|\nabla u|) + |\nabla u|^2 )+ C |  I_{r^2} |  \int_{Q_R} | \nabla f|^2 
\end{equation}
for any $r < R$ and concentric parabolic cylinders $Q_r, Q_R \Subset \O_I$. The second term of r.h.s. of \eqref{eq:jb_strong:2+_nablaX}, if not controlled by the first (weak existence) term on r.h.s., can be estimated as follows 
  \begin{equation}\label{eq:jb_strong:2+_nabla'X}
\esssup_{\tau \in [t_1, t_2] } \int_{B_R} | \nabla u |^2 (\tau)+   \int_{t_1}^{t_2}  \int_{B_R}  | \nabla  \V( \nabla u) |^2 \le \int_{B_{R_0}} |\nabla u (t_0) |^2 + \frac{C ( G(\vfi')}{({R_0}-R)^2}  \int_{t_1}^{t_2}   \int_{B_{R_0}}  \vfi (|\nabla u|)  
\end{equation}
for any nonnegative  $t_1 \le t_2$ from $I$ (in particular,  $t_1 = a$, possibly initial datum of the associated initial-boundary value problem),   $R < R_0$ and concentric balls $B_R, B_{R_0} \Subset \O_I$. 

\end{theo}
\section{Auxiliary results}\label{sec:aux}
In this section we provide a starting point for the proof of Theorems \ref{th1}, \ref{th2} via difference quotients, namely the problem reformulation that contains distributional time derivatives of a weak solution. The technicalities here are taken from \cite{ElmMes05}.

For brevity, we will use $S$ to denote either a domain being a subset of $\O$ or of $\O_I$, depending on the context.

Let us introduce the following notation

\[S^\delta := \{ s \in S | \; {\rm dist}(s, \partial S) \ge \delta \}\]
Here and in what follows we will denote also $[a, b - h_0]$ by $I_0$. Moreover
\[T_s g (x,t) := g (s +x,t),\]
\[  \Delta_h g (x,t) :=  g (x, t+h) - g(x, t), \qquad \Delta^{s} g (x,t) :=  T_s g (x,t) - g (x, t).\] 

Now let us introduce Steklov averages. In this context, one can additionally consult introduction of \cite{DiB93}, Naumann,  Wolf \&   Wolff \cite{NauWolWol98} or Chapter II, \textsection 4 of \cite{LadSolUra68} by Ladyzhenskaya, Solonnikov \& Ural'tseva. 
\begin{mydef}
Fix interval $I = (a, b)$. For $g \in L^1 ( \Omega_I )$, $h \in (0, b)$ its {'Steklov average'} $f_h$ is
\begin{equation}
g_h (t,x):= \left\{\begin{aligned}
&\dashint_{t}^{t+h} g (\tau, x) d \tau &\text{ for } t \in (a, b-h),\\
&0 &\text{ otherwise. }  \end{aligned}\right.
 \end{equation}
\end{mydef}
The Steklov averages have good mollification properties, presented below.
\begin{lem}\label{lem:jb_weak:stekl_prop}
Take $v \in L^\vfi (\O_I)$. Then
\begin{equation}\label{eq:jb_weak:stekl_prop:int}
\int_{\O_I}  \vfi(|v_h|)  \le \int_{\O_I}  \vfi( |v|), \qquad v_h \stackrel{h \to 0^+}{\longrightarrow} v \quad \text{ in } L^\vfi ({\O_I} ).
\end{equation}
The restriction of  $v_h$ to $(a, b - h) $ has the weak time derivative, for which
\begin{equation}\label{eq:jb_weak:stekl_prop:der}
(v_h),_t (x,t) = h^{-1}  \Delta_h v (x,t)
 \end{equation}
a.e. in $(a, b - h) \times \O$.\newline
Take $v \in C (I; L^2 (\Omega) )$.  At any $t \in (a, b)$
 \begin{equation}\label{eq:jb_weak:stekl_prop:c}
v_h (t)  \stackrel{h \to 0^+}{\longrightarrow} v (t) \quad \text{ in } L^2 (\Omega).
 \end{equation}

\end{lem}
\begin{proof}
This Lemma is a combination of  Lemma 3.2 in  \cite{DiB93}, Chapter I.3-(i) and results of pages 240 -- 241 of  \cite{NauWolWol98} with an exception of \eqref{eq:jb_weak:stekl_prop:int}. There, the inequality holds by the Jensen inequality with the~Tonelli theorem and the convergence by a density argument, available thanks to $\Delta_2$-growth of $\vfi$.
\end{proof}
Now we use Steklov averages to reformulate the notion of the local weak solution.
Choose in \eqref{eq:jb_weak:e_ws} $t_1 := t, t_2 := t+h$ and a test function $v$,  which is time-independent  on interval $(t, t+h)$, to obtain
\[
\int_\O \Delta_h u (t ) \cdot v + \int_\O  \int^{t+h}_{t}  \A (\D u)\!:\! \D v = \int_\O  \int^{t+h}_{t}  f  v
\]
Multiplication of the above formula by $h^{-1}$ and \eqref{eq:jb_weak:stekl_prop:der} provide us with the following identity for a local weak solution to \eqref{eq:jb_structure:weak}, that has a time derivative. Namely, for any $t, h$ such that $a \le t < t+h \le b$
\begin{equation}\label{eq:jb_weak:steklov_ws} 
\int_\O (u_h (t )),_t  \cdot \,v  + (\A (\D u))_h (t) : \D v  =  \int_\O  f_h  v.
\end{equation} 
Observe that terms in   \eqref{eq:jb_weak:steklov_ws} are meaningful  for any $v  \in   W^{1,\vfi}_{0}(\O) \cap L^2 (\O) $.
Integration in time yields 
\begin{equation}\label{eq:jb_weak:steklov_ws2} 
\int_{t_1}^{t_2} \int_{\O} (u_h ),_t  \cdot \,w  + (\A (\D u))_h  : \D w  = \int_{t_1}^{t_2} \int_\O  f_h  w
\end{equation} 
for any $w \in   W^{1,\vfi}_{x,0}(\O_I) \cap L^2 (\O_I) $ (via a density argument again) and any $t_1, t_2, h$ such that $a \le t_1 < t_2+h \le b$. This allowed us to widen the admissible class of test functions over these with no time derivatives. 

Recall that we denote $[a, b - h_0]$ by $I_0$. We are thus prepared to show

\begin{lem}\label{lem:distr} 
Let us now fix a small $h_0 >0 $. A local weak solution $u$ of \eqref{eq:jb_structure:weak} on $\O_I$  has a distributional time derivative  $u_t\in { W^{-1,\vfi^*}_{x} (\O_{I_0})}$ that enjoys
\[
 | u_t  |_{ (W^{-1,\vfi^*}_{x} + L^2)(\O_{I_0})} \le C (G (\vfi'))  \int_{\O_{I}} \left( 1 +\vfi(|\D u|) + |f|^2 \right).
 \]
and it satisfies on $\O_{I_0}$
\[
\langle  u_t, w \rangle_{{ (W^{-1,\vfi^*}_{x} + L^2)(\O_{I_0})} ,\,  W^{1,\vfi}_{x,0}(\O_{I_0}) \cap L^2 (\O_{I_0}) } +  \int_{\Omega_{I_0} }   \A ( \D u ) \D w  =   \int_{\Omega_{I_0} }  f w
\]
for any $w \in   W^{1,\vfi}_{x,0}(\O_{I_0}) \cap L^2 (\O_{I_0}) $.
\begin{proof} 
Let us consider next $h \le h_0$. 
Formula \eqref{eq:jb_weak:steklov_ws2} gives
\begin{multline}
\sup_{|w|_{W^{1,\vfi}_{x,0}(\O_{I_0}) \cap L^2 ((\O_{I_0})} \le 1} \left|  \int_{\O_{I_0}} \, (u_h),_t \cdot w \right|   \le \sup_{|w|_{W^{1,\vfi}_{x,0}(\O_{I_0}) \cap L^2 ((\O_{I_0})} \le 1} \left|    \int_{\O_{I_0}}  (\A (\D u))_h  \D w - f_h w \right| \le \\   C (G(\vfi'))  \int_{\O_I} \left(\vfi (|\D u| ) + 1 + |f|^2 \right).
\end{multline}
For the last inequality we used the Fubini theorem, Assumption \ref{ass:jb_weak:growth_gen} and Lemma \ref{lem:jb_weak:goodphi}. The above inequality and separability of $W^{1,\vfi}_{x,0}(\O_{I_0})$ (by definition, it is a  the norm closure of  $\DD (S)$) implies that there exists a $g \in 
{ W^{-1,\vfi^*}_{x} (\O_{I_0}) }$, being the sequential $*$-weak limit of a subsequence $(u_{h_n}),_t $. By l.w.s.c., it satisfies
\begin{equation}\label{eq:Ina}
 | g |_{ W^{-1,\vfi^*}_{x} (\O_{I_0}) } \le C (G (\vfi'))  \int_{\O_I} \left( 1 +\vfi(|\D u|)+ |f|^2   \right) .
\end{equation}
For $w \in \DD (\O_{I_0})$ we see via  \eqref{eq:jb_weak:steklov_ws2}  that
\begin{equation}\label{eq:Inb}
\langle (u_{h_n} ),_t , w \rangle_{W_x^{-1,\vfi^*} + L^2,\,  W_{x,0}^{1,\vfi} \cap L^2 } =  \int_{\O_{I_0}} (u_{h_n} ),_t  \cdot \,w  = \int_{\O_{I_0}} -\, (u)_{h_n}  \cdot w,_t ,
\end{equation}
where the l.h.s. converges to $\langle g, w \rangle_{W_x^{-1,\vfi^*},\,  W_{x,0}^{1,\vfi} } $, whereas r.h.s. goes to $  \int_{\O_{I_0}} -\,u  \cdot w,_t $. It means that by definition $g = u_t$ in 
$\DD' (\O_{I_0})$. Since by equation   $\int_{\O_{I_0}} -\,u  \cdot w,_t = - \int_{\O_{I_0} }   \A ( \D u ) \D w + f w$, we get from the limit of \eqref{eq:Inb}
\[
\langle u_t , w \rangle_{W_x^{-1,\vfi^*} + L^2 ,\,  W_{x,0}^{1,\vfi} \cap L^2} + \int_{\O_{I_0} }   \A ( \D u ) \D w = \int_{\O_{I_0} } f w
\]
for any $w \in   (W^{1,\vfi}_{x,0} \cap L^2) (\O_{I_0})$.
 \end{proof}
\end{lem}
An analogous result holds for space-differences. Before stating it, let us explain how we understand $\Delta^{s} h$, when $h$ is merely a linear functional. This interpretation will be needed below for $(\Delta^{s} u)_t$, where
\[ u_t  \in { W^{-1,\vfi^*}_{x} (\O_{I}) }  \quad \text{ and } \quad u \in { W^{1,\vfi}_{x,0} (\O_{I}) } ,\]
 so let us focus on this case. Fix $\delta>0$. Let us take any $w \in  W^{1,\vfi}_{x,0} (\Omega^\delta_{I^\delta} )$ and extend it with zero to  $w_0 \in W^{1,\vfi}_{0} (\Omega^\frac{\delta}{2}_{I^\delta}  ) $. We  define for any  $s \in \er^d, |s| \le \frac{\delta }{2}$ 
\begin{equation}\label{eq:jb_weak:dual}
\langle  (\Delta^{s} u)_t, w \rangle_{(W^{-1,\vfi^*} + L^2) (\O^\delta_{I^\delta} ),  (W^{1,\vfi}_{0} \cap L^2) (\O^\delta_{I^\delta} ) }  := -  \langle  u_t, \Delta^{-s} w_0 \rangle_{(W^{-1,\vfi^*} + L^2) (\O^\frac{\delta}{2}_{I^\delta} ),  (W^{1,\vfi}_{0} \cap L^2) (\O^\frac{\delta}{2}_{I^\delta} )}
\end{equation}
Taking above smooth enough functions,  we see that  $(\Delta^s u)_t$ defined by \eqref{eq:jb_weak:dual} is truly the generalized time derivative of $\Delta^{s} u$.

Formula  \eqref{eq:jb_weak:dual} and the fact that time derivatives commute with space differences allow us to prove along lines of Lemma \ref{lem:distr} the following result.

\begin{lem}\label{cor:jb_strong:intro} Take a local weak solution $u$  of \eqref{eq:jb_structure:weak} on $\O_I$.
Fix small $\delta > 0$ such that $\O^\delta $ is nonempty. For almost any $s \in \er^d,\, |s| \le \frac{\delta }{2}$

\[
 |  (\Delta^{s} u)_t  |_{ (W^{-1,\vfi^*}_{x} + L^2) (\O^\delta_{I_0})} \le C (G (\vfi'))  \int_{\O_I} \left( 1 +\vfi(|\D u|)   + |f|^2 \right).
 \]
and
\[
\langle  (\Delta^{s} u)_t, w \rangle_{(W^{-1,\vfi^*} + L^2) (\O^\delta_{I_0})  ,  (W^{1,\vfi}_{0} \cap L^2) (\O^\delta_{I_0}) } + \int_{\O^\delta_{I_0}}  \Delta^{s} \! \left( \A ( \D u ) \right)    \D w  =   \int_{\O^\delta_{I_0}}  (\Delta^{s} f) w
\]
for any $w \in   (W^{1,\vfi}_{x,0} \cap L^2) (\O^\delta_{I_0})$
\end{lem}

\section{Proofs of main results}\label{sec:prf}

In the first subsection, we present a formal argument for validity of our theorems \emph{i.e.} a priori testing with $\divv ((\nabla u) \psi^2)$. The subsequent sections contain its rigorization via the technique from \cite{DieEtt08} by Diening and Ettwein, which we modify for the evolutionary, symmetric-gradient case and improve. This improvement  lies in the fact that we do not need to resort to a covering argument and to Giaquita-Modica-type lemma for estimates with growing supports (but, at the same time, we are not interested in the Musielak-Orlicz growths $\varphi (x, | S |)$).

\subsection{Prologue. A priori estimates}
Assume that a solution $u$ to \eqref{eq:jb_structure:weak} or to  \eqref{eq:jb_structure:weakX} is smooth. We test it with  $\divv ((\nabla u) \psi^2)$, where $ \psi \in \test (I; C^\infty_0 (\O) )$ with values in $[0,1]$ will be precised in what follows. Hence in the case \eqref{eq:jb_structure:weak} we get
 \begin{multline}\label{eq:jb_strong:apriori_1}
\sup_{t \in [t_1, t_2]} \frac{1}{2} \int_{\O} | \nabla u \psi   |^{2} (t)+  \int_{t_1}^{t_2} \int_{\O}  \partial_{x_i} ( \A ( \D u))  \!:\! \D u,_{x_i} \psi^{2}\le \\
  \int_{t_1}^{t_2} \int_{\O} |  \nabla u |^2 |  \psi,_t \!|_\infty  \psi  + 2 \left|  \partial_{x_i} ( \A ( \D u)) \right|  |  u,_{x_i} \!|    |\nabla \psi |_\infty  \psi +| f_{x_i} u,_{x_i} | \psi^{2}  +  \frac{1}{2} \int_{\O} |  \nabla u \psi   |^{2} (t_1)
 \end{multline}
 and an analogous estimate with $\nabla u$ in place of $\D u$ holds for   \eqref{eq:jb_structure:weakX}.
Assumption \ref{ass:jb_weak:growth_gen} gives \eqref{eq:jb_structure:diffa1}, i.e.
\[
\begin{aligned}
\partial_s (\A(P)) : \partial_s P  & \ge  c\, \vfi'' (|P|) |\partial_s P|^2, \\
| \partial_s (\A(P)) |& \le C\, \vfi'' (|P|) |\partial_s P|,
\end{aligned}
\]
 which applied in the preceding estimate  gives  via the Schwarz inequality
 \begin{multline}\label{eq:jb_strong:apriori_3}
\sup_{t \in [t_1, t_2]} \frac{1}{2} \int_{\O} | \nabla u \psi   |^{2} (t)  +  \int_{t_1}^{t_2}  \int_{\O}    \vfi'' (|  \D u|) |  \nabla \D u |^2  \psi^2  \le \\
\int_{t_1}^{t_2} \int_{\O} |  \nabla u |^2    |  \psi,_t \!|_\infty     \psi 
+ C   |\nabla \psi |^2_\infty \int_{t_1}^{t_2} \int_{\O}  \vfi'' (|  \D u|)   | \nabla u |^2 +  \frac{1}{2} \int_{\O} |  \nabla u \psi   |^{2} (t_1)  + \eps \int_{t_1}^{t_2} \int_{\O} | \nabla u |^2   \psi^2 + \frac{1}{4 \eps} \int_{t_1}^{t_2} \int_{\O} | \nabla f |^2
 \end{multline}
again, an analogous estimate with $\nabla u$ in place of $\D u$ holds for   \eqref{eq:jb_structure:weakX}. Now we split considerations for \eqref{eq:jb_structure:weak} and for \eqref{eq:jb_structure:weakX}.
 \subsubsection*{Symmetric gradient case  \eqref{eq:jb_structure:weak}.}
In this case we assume that  $\vfi''$ is almost increasing. Consequently $\vfi'' (|  \D u|)   \le C \,  \vfi'' (|  \nabla u|)$. The good $\vfi'$ property and \eqref{eq:jb_weak:goodphi} give $ \vfi'' (|  \nabla u|)   | \nabla u |^2 \le C (G (\vfi') ) \, \vfi(| \nabla u |) $. Again,  almost increasingness of $\vfi''$ implies also
\begin{equation}\label{eq:jb_strong:apriori_pt}
|  \nabla u  |^2 \le (\delta + |  \nabla u  |)^2 \le  \frac{2C}{\vfi''(\delta+ |  \nabla u  |)} \vfi( \delta + | \nabla u |) \le  \frac{2C}{\vfi''(\delta)} \left( \vfi( \delta )+  \vfi(  | \nabla u |) \right),
\end{equation}
for any $\delta$ such that $\vfi''(\delta) \neq 0$. The second inequality above follows from $\Delta_2$ condition and convexity. 
This and choosing $\vfi$ to be a space-time cutoff  function that vanishes outside $B \times {[t_1, t_2]}$ and is $\equiv 1$ on smaller $B' \times  {[t'_1, t'_2]}$ in \eqref{eq:jb_strong:apriori_3} yields
 \begin{multline*}
\sup_{t \in [t'_1, t'_2]}  \int_{B'} |  \nabla u  |^2 (\tau)+  \int_{t'_1}^{t'_2} \! \int_{B'}  \vfi'' (|  \D u|) |  \nabla \D u |^2 \le \\
 C (G (\vfi')) \left( 1+ \frac{1}{\vfi''(\delta )} \right)   (|  \psi,_t \! |_\infty   + |\nabla \psi |^2_\infty)  \int_{B \times {[t_1, t_2]}}  \left( \vfi (|  \nabla u|) + \vfi(\delta ) + |f|^2 \right).
 \end{multline*}
Using on l.h.s. above  \eqref{eq:jb_structure:diffa15} with $ C \vfi'' (|  \D u|) \id_{\{|  \D u| \ge \delta\} } \ge \vfi'' (\delta)  \id_{ \{|  \D u| \ge \delta \}} \ge 0$ we arrive at
 \begin{multline*}
\sup_{t \in [t'_1, t'_2]}  \int_{B'} |  \nabla u  |^2 (\tau)+  \int_{t'_1}^{t'_2} \! \int_{B'}  \vfi'' (|  \D u|) |  \nabla \D u |^2 +  \vfi'' (\delta) |  \nabla^2 u |^2  \id_{\{|  \D u| \ge \delta \}}  \le \\
 C (G (\vfi')) \left( 1+ \frac{1}{\vfi''(\delta )} \right)   (|  \psi,_t \! |_\infty   + |\nabla \psi |^2_\infty)  \int_{B \times {[t_1, t_2]}}  \left( \vfi (|  \nabla u|) + \vfi(\delta ) + |\nabla f|^2 \right).
 \end{multline*}
For the case $\vfi'' (0) >0$ we may take there $\delta=0$, $\vfi(\delta) = 0$.
 \subsubsection*{Full gradient case  \eqref{eq:jb_structure:weakX}.}
 The above estimate holds also for the full-gradient case. Here, however, we in fact  do not need to assume any additional growth restrictions on $\vfi''$, like the almost-increasingness before. We start at \eqref{eq:jb_strong:apriori_3} with $\nabla u$ in place of $\D u$. In the first step, we choose there  $\psi$ to be a pure space-cutoff function, i.e. $\psi =   \psi_0 \equiv 1$ on $[t_1, t_2]$, then the r.h.s. part containing $  |  \psi,_t \!|_\infty   $ vanishes\footnote{Compare Remark \ref{rem1}. }. This and choice of $\eps = \frac{1}{4 (t_2 - t_1)}$ allows us to gain control over the quadratic term
\begin{multline}\label{eq:jb_strong:apriori_ptX}
\sup_{t \in [t_1, t_2]} \int_{B'} | \nabla u   |^{2} (t)    \le C   |\nabla \psi_0  |^2_\infty \int_{t_1}^{t_2} \int_{B}  \left( \vfi'' (|  \nabla u|)   | \nabla u |^2 + |\nabla f|^2   \right)+  C   \int_{B} |  \nabla u     |^{2} (t_1)  \le \\
 C   |\nabla \psi_0  |^2_\infty \int_{t_1}^{t_2} \int_{B}  \left( \vfi (|  \nabla u|)  + |\nabla f|^2 \right) + C   \int_{B} |  \nabla u     |^{2} (t_1) ,
\end{multline}
where the second inequality follows from the good $\vfi'$ property\footnote{The first inequality is available also for the  \eqref{eq:jb_structure:weak} case, with  $\vfi'' (| D u|)   | \nabla u |^2$ on r.h.s. However, the second inequality does not hold for general growths $\vfi$.}. Observe that, since it works for any $t_1 < t_2$, we can choose there $t_1 = a$ (possibly the initial value).

In the second step, we obtain from \eqref{eq:jb_strong:apriori_3}, similarly as in the symmetric-gradient case, that
\[
\sup_{t \in [t'_1, t'_2]}  \int_{B''} |  \nabla u  |^2 (\tau)+  \int_{t'_1}^{t'_2} \! \int_{B''}  \vfi'' (|  \D u|) |  \nabla \D u |^2 \le \\
 C (G (\vfi')) (|  \psi,_t \! |_\infty   + |\nabla \psi |^2_\infty)  \int_{B' \times {[t_1, t_2]}}  \left( \vfi (|  \nabla u|) + |  \nabla u|^2 + |\nabla f|^2 \right)
\]
and we can control the quadratic term of its r.h.s. by the l.h.s. of the previous inequality. Finally, on r.h.s. we can replace the full gradient with the symmetric one via Lemma \ref{lem:korn} (Korn's inequality).
 \subsubsection*{A remark on the missing case in the symmetric gradient case  \eqref{eq:jb_structure:weak}. }
One can try to close the symmetric-gradient version of  \eqref{eq:jb_strong:apriori_ptX}, i.e.
\[
\sup_{t \in [t_1, t_2]} \int_{B'} | \nabla u   |^{2} (t) +  \int_{t'_1}^{t'_2} \! \int_{B'}  \vfi'' (|  \D u|) |  \nabla \D u |^2    \le C   |\nabla \psi_0  |^2_\infty \int_{t_1}^{t_2} \int_{B}  \left( \vfi'' (|  \D u|)   | \nabla u |^2 + |\nabla f|^2   \right)+  C   \int_{B} |  \nabla u     |^{2} (t_1).
\]
For instance, for $\vfi'' $ almost-decreasing the r.h.s. can be controlled by the quadratic growths of $| \nabla u   |$. This and Giaquinta-Modica-type lemma for increasing supports allows to close the estimate, provided the r.h.s. controls the quadratic growths. To quantify this, one would need however to resort to Boyd indices. We feel that this approach may be too close to the polynomial case to be interesting. Another option is to use another estimate than \eqref{eq:jb_strong:apriori_1}, namely the one that involves $ \left|  ( \A ( \D u)) \right|  |   \partial^2_{x_i} u | $ instead of  $\left|  \partial_{x_i} ( \A ( \D u)) \right|  |  u,_{x_i} \!| $ on its r.h.s. Proceeding like this, in order to close estimates one needs to deal with derivatives and not with differences.

Now let us proceed with the rigorous proof. We provide it for the symmetric-gradient case \eqref{eq:jb_structure:weak} and finally we comment the full-gradient case \eqref{eq:jb_structure:weakX}.
\subsection{Proof of Theorem \ref{th1}}
\subsubsection{Step 1. Initial inequality for differences} 
In this step, we  obtain a starting point inequality for a first-order Caccioppoli estimate for space differences $\Delta^{s} u$. Since Definition \ref{def:weak} of a weak solution does not involve a local energy estimate and we are not allowed to immediately derive a $L^\infty (L^2)$ estimate from our notion of a weak solution, we will resort to certain approximations.
Recall that the interval $I = [a, b]$ and that $I_0 = [a, b - h_0]$. Let us choose any $\tau_0 \le \tau$ from $I_0$ and fix it. 

Fix small $\eps > 0$ such that $\O^\eps $ is nonempty. 
Using a  local weak solution $u$ of \eqref{eq:jb_structure:weak}, we write the formula for differences, where $s \in \er^d,\, |s| \le \frac{\eps }{2}$
\[
\int_{\O^\eps} (\Delta^{s} u) (t ) \cdot w (t){ \Big|}^{\tau}_{\tau_0 } + \int^{\tau}_{\tau_0 } \int_{\O^\eps} -\, (\Delta^{s} u) \cdot w,_t + \, \Delta^{s} (\A (\D u))  \D w =  \int^{\tau}_{\tau_0 } \int_{\O^\eps} (\Delta^{s} f) w.
\]
Now for $w = \tilde w \eta^k \sigma^{2k}$, with $\eta \equiv \eta (x)$ in $\test (\O^\delta)$, $\sigma \equiv \sigma (t)$  in $\test (I_0)$ and $k \in \en$ to be decided later, we have
\begin{equation}\label{eq:apprDa}
\int_{\O^\eps} (\Delta^{s} u) (\tau) \eta^k \cdot \tilde w (\tau)  \sigma^{2k} + \int^{\tau}_{\tau_0 } \int_{\O^\eps} -\, (\Delta^{s} u)  \eta^k \cdot  ( \tilde w \sigma^{2k}),_t + \, \Delta^{s} (\A (\D u))  \D ( \tilde w \eta^k) \sigma^{2k}=  \int^{\tau}_{\tau_0 } \int_{\O^\eps}  (\Delta^{s} f)   \tilde w \eta^k \sigma^{2k}
\end{equation}
Lemma  \ref{cor:jb_strong:intro} and existence class for $u$ imply that
\[
 (\Delta^{s} u)_t \in (W^{-1,\vfi^*}_{x} + L^2) (\O^\eps_{I_0}), \qquad  (\Delta^{s} u) \in (W^{1,\vfi}_{x} \cap L^2) (\O^\eps_{I_0}),
\]
hence
\[
 ((\Delta^{s} u) \eta^k )_t \in  (W^{-1,\vfi^*}_{x} + L^2) (\O^\eps_{I_0}), \qquad  (\Delta^{s} u)   \eta^{k}  \in (W^{1,\vfi}_{x, 0} \cap L^2) (\O^\eps_{I_0}).
\]
Therefore \cite{ElmMes05} provides us with a smooth approximations $\phi_n$ such that
\[ 
\begin{aligned}
(\phi_n),_t &\to (\Delta^{s} u)_t \eta^k \text{ in }  (W^{-1,\vfi^*}_{x} + L^2)((\O^\eps_{I_0}))   \\
 \phi_n &\to (\Delta^{s} u) \eta^k \text{ in } (W^{1,\vfi}_{x, 0}\cap L^2)((\O^\eps_{I_0})) \\
  \phi_n &\to (\Delta^{s} u)  \eta^k  \text{ in } C (I_0, L^2(\O^\eps))
\end{aligned} \]
where the convergence are in the sense of norms; the former of norms of  the representation \eqref{eq:modular}. The first two convergences are given as  Theorem 1 of \cite{ElmMes05},  since $\Delta_2$-condition gives us equivalence between the modular convergence and the strong one. The third convergence above follows from Lemma 3 in \cite{ElmMes05} (it is stated as a trace-type result, but the appropriate  convergence is given in the proof).

We put  $\tilde w = \phi_n$  in \eqref{eq:apprDa} and get
\begin{equation}\label{eq:apprD}
\int_{\O^\eps} (\Delta^{s} u) (\tau ) \eta^k \cdot  \phi_n (\tau)  \sigma^{2k} + \int^{\tau}_{\tau_0} \int_{\O^\eps} -\, (\Delta^{s} u)  \eta^k \cdot  (  \phi_n \sigma^{2k}),_t + \, \Delta^{s} (\A (\D u))  \D (  \phi_n \eta^k) \sigma^{2k}=  \int^{\tau}_{\tau_0 } \int_{\O^\eps}  (\Delta^{s} f)   \phi_n \eta^k \sigma^{2k}
\end{equation}
The first integral of l.h.s. of \eqref{eq:apprD} goes to 
\[
\int_{\O^\eps} |(\Delta^{s} u) (\tau)  \eta^k \sigma^{k} |^2
\]
as $n \to \infty$, in view of $C(L^2)$ convergence.
The second can be written as
\[
\int^{\tau}_{\tau_0 } \int_{\O^\eps} (-\, (\Delta^{s} u)  \eta^k + \phi_n) \cdot  (  \phi_n \sigma^{2k}),_t - \int^{\tau}_{\tau_0} \int_{\O^\eps} \phi_n \cdot  (  \phi_n \sigma^{2k}),_t  =
\]
\[
\int^{\tau}_{\tau_0 } \int_{\O^\eps} (-\, (\Delta^{s} u)  \eta^k + \phi_n) \cdot  (  \phi_n \sigma^{2k}),_t - \frac{1}{2} \int_{\O^\eps} |\phi_n (\tau) |^2 \sigma^{2k} (\tau)  - k \int^{\tau}_{\tau_0 } \int_{\O^\eps}   |\phi_n|^2 \sigma^{2k-1} \sigma,_t.
\]
Hence it tends to 
\[
- \frac{1}{2}  \int_{\O^\eps} |(\Delta^{s} u) (\tau)  \eta^k \sigma^{k} |^2 - \int^{\tau}_{\tau_0 } \int_{\O^\eps}  k |(\Delta^{s} u) \eta^k |^2 \sigma^{2k-1} \sigma,_t,
\]
due to $W^{-1,\vfi^*}_{x} + L^2$ boundedness and $W^{1,\vfi}_{x, 0}\cap L^2$ convergence of  $\phi_n$.
Finally, by the same token combined with Assumption \ref{ass:jb_weak:growth_gen} on growth,  Lemma \ref{lem:jb_weak:goodphi} and convexity of $\vfi$,  the third term of l.h.s. of \eqref{eq:apprD} is in  limit
\[
 \int_{\Omega^\eps_{I_0} }    \Delta^{s} \! \left(  \A ( \D u  ) \right):   \left( \Delta^{s} \D u \right)   \eta^{2k} \sigma^{2k}  +
 2k  \Delta^{s} \! \left(  \A ( \D u  ) \right):  \left(  \Delta^{s}  u \; \hat \otimes \; \nabla \eta \right) \eta^{2k-1} \sigma^{2k},
\]
where we also used the fact that differences and weak derivatives commute; $\hat \otimes$ denotes symmetrization of the outer product $\otimes$, i.e. $(a \hat \otimes b)_{ij} := \frac{1}{2} (a^i b^j + a^j b^i)$. 

Summing up, we arrive from \eqref{eq:apprD} at
 \begin{multline}\label{eq:jb_strong:intro2}
 \frac{1}{2}  \int_{\O^\eps} \left|  \left(\Delta^{s} u \right)  \eta^k \sigma^k \right|^2 (\tau) +\int_{a}^\tau \int_{\Omega^\eps }  \Delta^{s} \left(  \A ( \D u  ) \right):   \left( \Delta^{s} \D u \right)   \eta^{2k} \sigma^{2k}  \le \\
 2k \int_{\tau_0}^\tau  \int_{\Omega^\eps } \left|  \Delta^{s} \left(  \A ( \D u  ) \right) \right| \left|   \Delta^{s} u \right| |\nabla \eta | \eta^{2k-1} \sigma^{2k} +   k \int_{\tau_0}^\tau  \int_{\O^\eps}    |(\Delta^{s}  u) \eta^k|^2 \sigma^{2k-1} |\sigma,_t| \\
 + \eps \int_{\tau_0}^\tau  \int_{\Omega^\eps }  \left|  \left(\Delta^{s} u \right)  \eta^k \sigma^k \right|^2 + \frac{1}{4 \eps}  \int_{\tau_0}^\tau  \int_{\Omega^\eps }  \left|  \left(\Delta^{s} f \right)\right|^2.
\end{multline}
Now we use Assumption  \ref{ass:jb_weak:growth_gen}. It gives via \eqref{eq:app:shifted:eq} the pointwise estimate
\[ \left|  \Delta^{s} \left(  \A ( \D u  ) \right) \right|  \le C(G(\vfi')) \vfi'_{| \D u|} \left(| \Delta^{s} \D u | \right) \]
  for the r.h.s. of \eqref{eq:jb_strong:intro2}. 
In tandem with Proposition \ref{lem:jb_weak:mon_equiv}, Assumption \ref{ass:jb_weak:growth_gen} yields also the pointwise majorization 
\begin{equation}\label{eq:jb_strong:intro:n1}
\Delta^{s} \left(  \A ( \D u  ) \right)\!:\!  \left( \Delta^{s} \D u \right)    \ge \frac{1}{C(G(\vfi'))}  | \Delta^{s} \V ( \D u ) |^2 
\end{equation}
 for the l.h.s. of \eqref{eq:jb_strong:intro2}. 

Hence we have obtained that for almost any $s \in \er^d, \, |s| \le \frac{\eps }{2}$ holds 
  \begin{multline}\label{eq:jb_strong:intro:delta_id_x}
\frac{1}{2}  \int_{\O^\eps} \left|  \left(\Delta^{s} u \right)  \eta^k \sigma^k \right|^2 (\tau)+  \int_{\tau_0}^\tau \int_{\O^{\delta}}  \left| \left( \Delta^{s} \V ( \D u ) \right) \eta^k \sigma^k \right|^2  
\le 
\\
  C (G (\vfi')) k \int_{\tau_0}^\tau  \int_{\O^{\eps}}  \left( |\Delta^{s}u |^2   |\sigma,_t |_\infty + \vfi'_{|\D u (t)| } \left(|  \Delta^{s} \D u| \right) | \Delta^{s} u| | \nabla \eta |_\infty \right) \eta^{2k-1} \sigma^{2k-1}  + \\
   \eps \int_{\tau_0}^\tau  \int_{\Omega^\eps }  \left|  \left(\Delta^{s} u \right)  \eta^k \sigma^k \right|^2 + \frac{1}{4 \eps}  \int_{\tau_0}^\tau  \int_{\Omega^\eps }  \left|  \left(\Delta^{s} f \right)\right|^2
\end{multline}
 for any $\tau \in I$.

\subsubsection{Step 2. Estimate for differences without a growing support.} 

In this step, choice of large $k$ in cutoff functions $\eta^k, \sigma^k$ will allow us to obtain the same supports on the l.h.s. and r.h.s. of relevant estimates. Alternatively, one can use cutoff functions $\eta, \sigma$ and to deal with growing supports by a Giaquinta-Modica device. The latter approach is longer.

Fix $\eps > 0$ so small  that the \"ubercylinder $Q_{R+ \eps} \Subset \O_I$. All the following work happens in $Q_{R+ \eps}$. Next, take any concentric $Q_{\rho_1} \Subset Q_{\rho_2}$, such that $Q_{\rho_2} \subset Q_R$. We write estimate \eqref{eq:jb_strong:intro:delta_id_x} with $ s := l e_i$, with real $|l| \le \frac{\eps}{2}$, where $e_i$ denotes i-th canonical vector in $\er^d$. Let $\eta$, $\sigma$ cut off between $Q_{\rho_1}$ and $Q_{\rho_2}$, i.e. $\eta \sigma \equiv 1$  on $Q_{\rho_1}$ and vanishes outside $Q_{\rho_2}$. We choose these cutoff functions so that $|\nabla\eta|\leq C/(\rho_2-\rho_1)$ and $|\sigma_t|\leq C/(\rho_2-\rho_1)^2$. Hence
  \begin{multline}\label{eq:jb_strong:spatial:ini}
\frac{1}{2}  \int_{B_{\rho_2}} \left| \Delta^{le_i} u \sigma^k \eta^k \right|^2  (\tau)+  \int_{\tau_0}^\tau\int_{B_{\rho_2}}  \left| \Delta^{le_i} \V(\D u)\sigma^k\eta^k \right|^2 
\le 
\\
  C (G (\vfi')) k
\int_{\tau_0}^\tau\int_{B_{\rho_2}} \left| \Delta^{le_i} u \eta^k\right|^2\sigma^{2k-1}|\sigma_t|  +   C (G (\vfi')) k
\int_{\tau_0}^\tau\int_{B_{\rho_2}}  \vfi'_{|\D u| } \left(|\Delta^{le_i} \D u| \right) | \Delta^{le_i} u|  \eta^{2k-1}|\nabla\eta|\sigma^{2k}  +  \\
\eps_0 \int_{\tau_0}^\tau  \int_{\Omega^\eps }  \left|  \left(\Delta^{le_i} u \right)  \eta^k \sigma^k \right|^2 + \frac{1}{4 \eps_0}  \int_{\tau_0}^\tau  \int_{\Omega^\eps }  \left|  \left(\Delta^{le_i} f \right)\right|^2 :=   C (G (\vfi')) \int_{\tau_0}^\tau ( A + B + \eps C + F).
\end{multline}
Observe that  the differences in the formula \eqref{eq:jb_strong:spatial:ini} remain within the \"ubercylinder $Q_{R+ \eps} $. Recall that $T_s g (x) := g (s +x)$.
In order to deal with the the r.h.s. of \eqref{eq:jb_strong:spatial:ini}, let us observe first that absolute continuity along lines of a Sobolev function $u(t)$ implies that for almost every $z \in Q_{\rho_2}$, real $l, \;|l| \le \frac{\eps}{2}$ we have
\begin{equation}\label{eq:jb_strong:spatial:acl}
|\Delta^{l e_i} u (z)|^\alpha \le \left| \int_0^l u,_{x_i}  \! \circ T_{\lambda e_i} (z) d \lambda \right|^\alpha \le \dashint_0^l |l|^\alpha | \nabla u \circ T_{\lambda e_i} (z)|^\alpha d \lambda
\end{equation}
for $\alpha \ge 1$. By $\dashint_0^l $ we understand here and in the following $\frac{1}{|l|} \int_0^l $ for $ l \ge 0$ and   $\frac{1}{|l|} \int_{-l}^0 $ otherwise. We introduce $\frac{\eps}{2} \ge h \ge l$, which will be useful later.  

\paragraph{Dealing with $B$} 

Now let us consider quantity $B$ of \eqref{eq:jb_strong:spatial:ini}.  First we estimate it using \eqref{eq:jb_strong:spatial:acl} with $\alpha = 1$, to get
\begin{equation}\label{eq:jb_strong:spatial:ini_2-pk}
B  \le  \sigma^{2k}  \int_{B_{\rho_2}} \dashint_0^l \frac{|l|}{|h|} \underbrace{\vfi'_{|\D u| } \left(|\Delta^{le_i} \D u| \right)  \frac{|h|}{{\rho_2}-{\rho_1}}\eta^{2k-1} | \nabla u \circ T_{\lambda e_i} | }_{:= J} d \lambda
\end{equation}
Let us focus on $J$. We change shift with help of  \eqref{eq:app:shifted:change} of Lemma \ref{lem:app:shifted} to get
\[
J \le C (G(\vfi'))  \left[ \vfi'_{|\D u \circ T_{\lambda e_i} | } \left(|\Delta^{(\lambda e_i, 0)} \D u| \right) + \vfi'_{|\D u \circ T_{\lambda e_i} | } \left(|\Delta^{(l -\lambda) e_i} \D u \circ T_{\lambda e_i} | \right) \right]\eta^{2k-1}  \frac{|h|}{{\rho_2}-{\rho_1}} | \nabla u \circ T_{\lambda e_i} | 
\]
with $\lambda \in [0, l]$. Recall Corollary \ref{cor:jb_weak:goodphi_shifted}. It allows to use Young's inequality \eqref{eq:jb_weak:Y} for shifted $\eN$-functions and property \eqref{eq:jb_weak:phistar}, i.e. $\varphi_{a}^*\circ \varphi_{a}'   \le C (G (\vfi_{a}')) \varphi_{a} $, with a common bound $C (G (\vfi_{a}')) $ on constants. 
Here we also combine Proposition~\ref{prop:jb_structure:polyGrowth} with Lemma~\ref{lem:app:shifted} to extract $\eta^{2k-1}$ out of $\eN$-functions $\vfi_a^*$. Hence
\begin{multline}\label{eq:jb_strong:spatial:m1}
J \le \frac{\delta}{ 12 C_1} \eta^{(2k-1)q_1} \left[ \vfi_{|\D u \circ T_{\lambda e_i}  | } \left(|\Delta^{\lambda e_i}  \D u| \right) + \vfi_{|\D u \circ T_{\lambda e_i}  | } \left(|\Delta^{(l -\lambda) e_i}  \D u \circ T_{\lambda e_i}  | \right) \right] + \\
C (\delta, C_1, G(\vfi')) \vfi_{|\D u \circ T_{\lambda e_i}  | } \left( \frac{|h|}{{\rho_2}-{\rho_1}} | \nabla u \circ T_{\lambda e_i}  | \right)
\end{multline}
for a yet unspecified $C_1$. To obtain the inequality above we used also convexity of an $\eN$-function. Now we set $k\in \en$ so large that $(2k-1)q_1>2k$. By  $\vfi_{|P|} ( |P-Q|) \sim | \V(P) - \V(Q) |^2$ of Proposition \ref{lem:jb_weak:mon_equiv}, we majorize the $\delta$-part of the r.h.s. of \eqref{eq:jb_strong:spatial:m1} by
\begin{equation}\label{eq:jb_strong:spatial:m2}
\eta^{2k}\frac{\delta C(G(\vfi'))}{  C_1}  \left[  |\Delta^{\lambda e_i}  \V (\D u)|^2 +  |\Delta^{(l -\lambda) e_i}  \V \left(\D u \circ T_{\lambda e_i}  \right) |^2\right]
\le 3\eta^{2k} \delta |\Delta^{\lambda e_i}  \V ( \D u) |^2   + 2 \eta^{2k}\delta  |\Delta^{l e_i}  \V ( \D u) |^2
\end{equation}
where the second inequality follows from the choice $C_1:= C(G(\vfi'))$ and from the identity
\begin{equation}\label{eq:strong:spatial:eff}
\Delta^{(l-\lambda) e_i} \V (\D u \circ  T_{\lambda e_i} ) = \Delta^{l e_i}  \V (\D u)  - \Delta^{\lambda e_i}  \V (\D u).
\end{equation}
Concerning the last summand of r.h.s.  of \eqref{eq:jb_strong:spatial:m1}, we increase it by changing its shift from ${|\D u \circ T_{\lambda e_i}  | }$ to larger ${|\nabla u \circ T_{\lambda e_i}  | }$. It is admissible thanks to Proposition \ref{prop:jb_weak:shiftedO} that gives $ |a| \le |b| \implies \vfi_a (t) \le \vfi_b (t) $. This\footnote{Here is the only place where we essentially use the almost increasingness of $\vfi''$. It is an interesting question if one can similarly deal with the subquadratic case, i.e. for $\vfi''$ almost decreasing.}\label{ft1}  and \eqref{eq:jb_strong:spatial:m2} used in  \eqref{eq:jb_strong:spatial:m1} gives
\begin{equation}\label{eq:jb_strong:spatial:m3}
J \le 3\eta^{2k} \delta |\Delta^{\lambda e_i}  \V ( \D u) |^2   + 2\eta^{2k} \delta  |\Delta^{l e_i}  \V ( \D u) |^2 + 
C (\delta, G(\vfi')) \vfi_{|\nabla u \circ T_{\lambda e_i}  | } \left( \frac{|h|}{{\rho_2}-{\rho_1}} | \nabla u \circ T_{\lambda e_i}  | \right)
\end{equation}
We would like now to extract $\frac{|h|}{{\rho_2}-{\rho_1}} $ from the argument of the shifted $\vfi$ above, most desirably as $\frac{|h|^2}{|{\rho_2}-{\rho_1}|^2} $. Unluckily, to this end one needs to impose an additional relation between $|h|$ and ${\rho_2}-{\rho_1}$. Namely only for $|h| \le {\rho_2}-{\rho_1}$ we have $\vfi_{|a|} (\frac{|h|}{{\rho_2}-{\rho_1}} |a|) \le C (G (\vfi'))
\frac{|h|^2}{|{\rho_2}-{\rho_1}|^2} \vfi (|a|) $, see \eqref{eq:app:shifted:little2} of Lemma \ref{lem:app:shifted}. Using this information in \eqref{eq:jb_strong:spatial:m3} and next plugging the obtained estimate for $J$ into \eqref{eq:jb_strong:spatial:ini_2-pk} we end up with

\begin{multline}\label{eq:jb_strong:spatial:ini_3}
B \le 
4 \delta   \int_{B_{\rho_2}}  |\Delta^{l e_i}  \V ( \D u) |^2\eta^{2k} \sigma^{2k}  + 6\delta   \int_{B_{\rho_2}}  \eta^{2k}  \sigma^{2k} \dashint_0^l \frac{|l|}{|h|} |\Delta^{\lambda e_i}  \V ( \D u) |^2  d \lambda  \\
+  
C (\delta, G(\vfi'))  \frac{|h|^2}{({\rho_2}-{\rho_1})^2} \vfi \left( | \nabla u \circ T_{\lambda e_i}  | \right)
\end{multline}

\paragraph{Dealing with $A$} 
Now we estimate the first term on the right hand side of \eqref{eq:jb_strong:spatial:ini}, \emph{i.e.} $A$. Using \eqref{eq:jb_strong:spatial:acl} with $\alpha = 2$, we
have for almost every $|l| \le \frac{\eps}{2}$ 

\begin{equation}\label{eq:jb_strong:spatial:ev1}
 A   \le  \frac{C}{({\rho_2}-{\rho_1})^2}  \intop_{Q_{\rho_2} } \left| \Delta^{le_i} u \right|^2 \le  \frac{C}{({\rho_2}-{\rho_1})^2}  \intop_{Q_{\rho_2} }  \dashint_0^l |l|^2 | \nabla u \circ T_{\lambda e_i} |^2 d \lambda \le \frac{C |h|^2}{({\rho_2}-{\rho_1})^2}  \intop_{Q_{\rho_2} }  \dashint_0^l  | \nabla u \circ T_{\lambda e_i} |^2 d \lambda.
\end{equation}
In the last inequality above we have increased $l $ to $h$; recall that  $\frac{\eps}{2} \ge h \ge l$.
 
\paragraph{Putting together estimates for  $A$ and $B$} 
Plugging \eqref{eq:jb_strong:spatial:ini_3} for $B$ and \eqref{eq:jb_strong:spatial:ev1} for $A$ into \eqref{eq:jb_strong:spatial:ini} we have
  \begin{multline}\label{eq:jb_strong:spatial:ini:fini}
\frac{1}{2}  \int_{B_{\rho_2}} \left| \Delta^{le_i} u \sigma^k \right|^2  (\tau)\eta^{2k}+  \int_{\tau_0}^\tau \int_{B_{\rho_2}}  \left| \Delta^{le_i} \V(\D u)\sigma^k\eta^k \right|^2  \le \\
 4 \delta  \int_{\tau_0}^\tau \int_{B_{\rho_2}}  |\Delta^{l e_i}  \V ( \D u) \sigma^k\eta^k |^2 +  6 \delta  \int_{\tau_0}^\tau   \dashint_0^l \frac{|l|}{|h|}   \int_{B_{\rho_2}}   |\Delta^{\lambda e_i}  \V ( \D u)  \; \sigma^k\eta^k  |^2 d \lambda  \\
   +  
C (\delta, G(\vfi'))  \frac{|h|^2}{({\rho_2}-{\rho_1})^2}   \dashint_0^l  \int_{Q_{\rho_2}} {\text {\Large [}} \vfi \left(  | \nabla u \circ T_{\lambda e_i}  | \right) + | \nabla u \circ T_{\lambda e_i} |^2 {\text {\Large ]}}  d \lambda +\\ \eps_0 \int_{\tau_0}^\tau  \int_{B_{\rho_2} }  \left|  \left(\Delta^{le_i} u \right)  \eta^k \sigma^k \right|^2 + \frac{ |h|^2}{4 \eps_0}  \int_{\tau_0}^\tau  \int_{B_{\rho_2} }  \dashint_0^l  | \nabla f \circ T_{\lambda e_i} |^2 d \lambda.
\end{multline}

For the term in square brackets above we use the pointwise majorization \eqref{eq:jb_strong:apriori_pt} with $\delta = \delta_0$.  
Hence  we can estimate the last integral of  \eqref{eq:jb_strong:spatial:ini:fini} by

\[
 \left(1 + \frac{2C}{\vfi'' (\delta_0)} \right)   \dashint_0^l     \int_{Q_{\rho_2}}    \vfi \left(  | \nabla u \circ T_{\lambda e_i} | \right) +  \vfi( \delta_0 ) d \lambda \le   \left(1 + \frac{2C}{\vfi'' (\delta_0)} \right)    \int_{Q_{\rho_2+ \eps}} \left(   \vfi \left(  | \nabla u  | \right) +  \vfi( \delta_0 )  \right),
\]
where to obtain the  inequality we increase the domain of inner integration, as $\lambda $ varies from $0$ to $l \le h \le \frac{\eps}{2} $.
 This estimate  in  \eqref{eq:jb_strong:spatial:ini:fini}  written for the parabolic cylinder $Q_{\rho_2}$, together with  choice $\delta \le \frac{1}{8}, \eps_0 = \frac{1}{4( \rho_2 - \rho_1)^2} $, give
  \begin{multline}\label{eq:jb_strong:spatial:ini:fini'}
\sup_{t \in I_{\rho^2_2}} \int_{B_{\rho_2}} \left| \Delta^{le_i} u \sigma^k \eta^k \right|^2  (t) +   \int_{Q_{\rho_2}}  \left| \Delta^{le_i} \V(\D u)\sigma^k\eta^k \right|^2  \le \\
  24 \delta     \dashint_0^l \frac{|l|}{|h|}   \int_{Q_{\rho_2}}   |\Delta^{\lambda e_i}  \V ( \D u)  \; \sigma^k\eta^k  |^2 d \lambda 
       +  
C (\delta, G(\vfi')) \left(1 + \frac{1}{\vfi'' (\delta_0)} \right) \frac{|h|^2}{({\rho_2}-{\rho_1})^2}    \int_{Q_{\rho_2+ \eps}}    \vfi \left(  | \nabla u  | \right) +  \vfi( \delta_0 ) +\\
 |h|^2 ({\rho_2}-{\rho_1})^2    \int_{Q_{\rho_2+ \eps}} | \nabla f|^2.
\end{multline}

\subsubsection{Step 3. The homogenization of the bad part.} 
In order to get rid of the $\delta$ part of  r.h.s of \eqref{eq:jb_strong:spatial:ini:fini'}, we homogenize it as follows. Let us drop  all but the middle term of the  l.h.s. of  \eqref{eq:jb_strong:spatial:ini:fini'}  and apply to both sides of  the resulting inequality $\dashint_0^h d l$ to get
  \begin{multline}\label{eq:jb_strong:spatial:ini:gm?}
 \dashint_0^h   \int_{Q_{\rho_2}}  \left| \Delta^{le_i} \V(\D u)  \sigma^k \eta^k \right|^2  dl  \le 
 24 \delta \dashint_0^h  \frac{|l|}{|h|}  \dashint_0^l    \int_{Q_{\rho_2}}   |\Delta^{\lambda e_i}  \V ( \D u)  \sigma^k \eta^k |^2 d \lambda dl \\
   +  
C (\delta, G(\vfi')) \left(1 + \frac{1}{\vfi'' (\delta_0)} \right) \frac{|h|^2}{({\rho_2}-{\rho_1})^2}    \int_{Q_{\rho_2+ \eps}}    \vfi \left(  | \nabla u  | \right)    +  \vfi( \delta_0 ) +  |h|^2  ({\rho_2}-{\rho_1})^2     \int_{Q_{\rho_2+ \eps}} | \nabla f|^2
\end{multline}
To estimate r.h.s. above we use the following inequality for a nonnegative $g$, valid in view of Tonelli's theorem
\[
 \dashint_0^h  \frac{|l|}{|h|}  \dashint_0^l g (\lambda) d \lambda dl =   \frac{1}{|h|^2} \int_0^h   g (\lambda) \left( \int_0^h \id_{ \{| \lambda | \le | l | \}}  dl \right) d \lambda \le    \frac{1}{|h|^2} \int_0^h   g (\lambda)  |h| d \lambda =  \dashint_0^h g (\lambda) d \lambda.
\]
Hence, using the above inequality to estimate the middle summand of r.h.s. of  \eqref{eq:jb_strong:spatial:ini:gm?}, we obtain
  \begin{multline}\label{eq:jb_strong:spatial:ini:gm?2}
 \dashint_0^h   \int_{Q_{\rho_2}}  \left| \Delta^{\lambda e_i} \V(\D u)  \sigma^k \eta^k \right|^2  d \lambda  \le \\
24 \delta \dashint_0^h  \int_{Q_{\rho_2}}  |\Delta^{\lambda e_i}  \V ( \D u)  \sigma^k \eta^k |^2 d \lambda    +  
 \left(1 + \frac{1}{\vfi'' (\delta_0)} \right) \frac{ C (\delta, G(\vfi'))  |h|^2}{({\rho_2}-{\rho_1})^2}  \!  \int_{Q_{\rho_2+ \eps}}    \vfi \left(  | \nabla u  | \right)     +  \vfi( \delta_0 ) + \\
  |h|^2 ({\rho_2}-{\rho_1})^2     \int_{Q_{\rho_2+ \eps}} | \nabla f|^2.
\end{multline}
We add \eqref{eq:jb_strong:spatial:ini:gm?2} to  \eqref{eq:jb_strong:spatial:ini:fini'},  where we choose  $l := h$, obtaining
  \begin{multline}\label{eq:jb_strong:spatial:ini:gm?3}
 \sup_{\tau \in I_{\rho^2_2}}  \int_{B_{\rho_2}} \left| \Delta^{he_i} u \;  \sigma^k \eta^k \right|^2  (\tau)+  \int_{Q_{\rho_2}}  \left| \Delta^{he_i} \V(\D u)  \sigma^k \eta^k \right|^2 +   \dashint_0^h   \int_{Q_{\rho_2}}  \left| \Delta^{\lambda e_i} \V(\D u) \sigma^k \eta^k  \right|^2   d \lambda \le \\  24 \delta \left(   \int_{Q_{\rho_2}}  |\Delta^{h e_i}  \V ( \D u) \sigma^k \eta^k  |^2 +  \dashint_0^h     \int_{Q_{\rho_2}}   |\Delta^{\lambda e_i}  \V ( \D u) \sigma^k \eta^k  |^2 d \lambda \right).
   +  \\
 \left(1 + \frac{1}{\vfi'' (\delta_0)} \right) \frac{C (\delta, G(\vfi')) |h|^2}{({\rho_2}-{\rho_1})^2}    \int_{Q_{\rho_2+ \eps}} \left(    \vfi \left(  | \nabla u  | \right)      +  \vfi( \delta_0 )  \right)+  |h|^2  ({\rho_2}-{\rho_1})^2   \int_{Q_{\rho_2+ \eps}} | \nabla f|^2.
\end{multline}
For $\delta \le \frac{1}{48}$ it produces
  \begin{multline}\label{eq:jb_strong:spatial:ini:gm?3'}
 \sup_{\tau \in I_{\rho^2_1}}  \int_{B_{\rho_1}} \left| \Delta^{he_i} u \;  \sigma^k \eta^k \right|^2  (\tau)+ \frac{1}{2} \int_{Q_{\rho_1}}  \left| \Delta^{he_i} \V(\D u)  \sigma^k \eta^k \right|^2  \\ \le |h|^2  \left(1 + \frac{1}{\vfi'' (\delta_0)} \right) \frac{C (G(\vfi')) }{({\rho_2}-{\rho_1})^2}    \int_{Q_{\rho_2+ \eps}} \left(   \vfi \left(  | \nabla u  | \right) +  \vfi( \delta_0 ) \right) +  |h|^2  ({\rho_2}-{\rho_1})^2   \int_{Q_{\rho_2+ \eps}} | \nabla f|^2,
\end{multline}
where we have used that $ \sigma^k \eta^k \equiv 1$ on $Q_{\rho_1}$.
\subsubsection{Step 4. Obtaining the thesis for the symmetric case \eqref{eq:jb_structure:weak}} 
We have reached the inequality that  implies an estimate for directional difference quotients $ h^{-1} \Delta^{(he_i, 0)} f$. However to finish this proof, we need to take care of the following technicality: recall that to obtain \eqref{eq:jb_strong:spatial:ini_3} from \eqref{eq:jb_strong:spatial:m3} we needed to take out $\frac{|h|}{{\rho_2}-{\rho_1}} $ from the argument of the shifted $\vfi$, for which we imposed $|h| \le {\rho_2}-{\rho_1}$. Therefore let us first fix ${\rho_2} > {\rho_1}$. Now \eqref{eq:jb_strong:spatial:ini:gm?3'} is 
an uniform estimate for any  $h>0, |h| \le {\rho_2}-{\rho_1}$ on directional difference quotients. This implies the existence of a weak derivative with the same bound (by compactness of balls in $L^2 (B_r)$, formula for discrete integration by parts of a difference quotient and l.w.s.c. of  $L^2 $ norms). Hence \eqref{eq:jb_strong:spatial:ini:gm?3'} gives
\begin{multline}
 \esssup_{\tau \in I_{\rho_1^2}}  \int_{B_{\rho_1}} \left|\nabla  u   \right|^2  (\tau)+  \int_{Q_{\rho_1}}  \left| \nabla \V(\D u) \right|^2 \le \\
  \left(1 + \frac{1}{\vfi'' (\delta_0)} \right) \frac{C (G(\vfi')) }{({\rho_2}-{\rho_1})^2}    \int_{Q_{\rho_2+ \eps}}  \left(   \vfi ( | \nabla u  |)+  \vfi( \delta_0 ) \right)  +  ({\rho_2}-{\rho_1})^2  \int_{Q_{\rho_2+ \eps}}    | \nabla f|^2,
\end{multline}
for any $\rho < R$. We use above  $3 |\nabla \D u|^2 \ge |\nabla^2 u|^2$ to write full second-order gradient in the last term of l.h.s.  \\
In the above estimate there are no more difference quotients. Particularly, $\eps> 0$ becomes a free parameter. Sending $\eps \to 0$, 
we have the first estimate \eqref{eq:jb_strong:2+_nabla} of the thesis. Via Lemma \ref{lem:korn} (Korn's inequality), we get the symmetric-gradient version  \eqref{eq:jb_strong:2+_nablaK}.
The assumption for the last estimate allows us to take $\delta_0 = 0$. This and \eqref{eq:jb_structure:diffa3} allows us to write the latter inequality \eqref{eq:jb_strong:2+_nabla'} of the thesis. Theorem \ref{th1} is proved.
\qed

\subsection{Proof of Theorem \ref{th2}} 
Observe that here we can arrive to the analogue of \eqref{eq:jb_strong:spatial:ini:fini}, without the assumption that $\vfi''$ is almost increasing. It was needed up to \eqref{eq:jb_strong:spatial:ini:fini} only to change the shift in \eqref{eq:jb_strong:spatial:m1}
\[
\text{from } \quad \vfi_{|\D u \circ T_{\lambda e_i}  | } \left( \frac{|h|}{{\rho_2}-{\rho_1}} | \nabla u \circ T_{\lambda e_i}  | \right) \quad \text{ to } \quad  \vfi_{|\nabla u \circ T_{\lambda e_i}  | } \left( \frac{|h|}{{\rho_2}-{\rho_1}} | \nabla u \circ T_{\lambda e_i}  | \right) 
\]
now we have full gradients, so no shift change is needed. Therefore we get the analogue of  \eqref{eq:jb_strong:spatial:ini:fini}, \emph{i.e.}
\begin{multline}\label{eq:jb_strong:spatial:ini:finiX}
\frac{1}{2}  \int_{B_{\rho_2}} \left| \Delta^{le_i} u \sigma^k \right|^2  (\tau)\eta^{2k}+  \int_{\tau_0}^\tau \int_{B_{\rho_2}}  \left| \Delta^{le_i} \V(\nabla u)\sigma^k\eta^k \right|^2  \le \\
 4 \delta  \int_{\tau_0}^\tau \int_{B_{\rho_2}}  |\Delta^{l e_i}  \V ( \nabla u) \sigma^k\eta^k |^2 +  6 \delta  \int_{\tau_0}^\tau   \dashint_0^l \frac{|l|}{|h|}   \int_{B_{\rho_2}}   |\Delta^{\lambda e_i}  \V ( \nabla u)  \; \sigma^k\eta^k  |^2 d \lambda  \\
   +  
C (\delta, G(\vfi'))  \frac{|h|^2}{({\rho_2}-{\rho_1})^2}   \dashint_0^l  \int_{Q_{\rho_2}} {\text {\Large [}} \vfi \left(  | \nabla u \circ T_{\lambda e_i}  | \right) + | \nabla u \circ T_{\lambda e_i} |^2 {\text {\Large ]}}  d \lambda + \\
\eps_0 \int_{\tau_0}^\tau  \int_{B_{\rho_2} }  \left|  \left(\Delta^{le_i} u \right)  \eta^k \sigma^k \right|^2 + \frac{ |h|^2}{4 \eps_0}  \int_{\tau_0}^\tau  \int_{B_{\rho_2} }  \dashint_0^l  | \nabla f \circ T_{\lambda e_i} |^2 d \lambda
\end{multline}
This gives the first estimate  \eqref{eq:jb_strong:2+_nablaX}, along the remainder of the proof of Theorem \ref{th1},  as long as $\int_{Q_{\rho_2 + \eps}} | \nabla u|^2$ is meaningful. 

Now, let us take in the full-gradient-analogue of  \eqref{eq:jb_strong:spatial:ini} a time-independent (on $[t_1, t_2]$) cutoff function $\equiv \eta (x)$. It yields for any $\tau \in [t_1, t_2]$ 
\begin{multline*}
\frac{1}{2}  \int_{B_{\rho_2}} \left| \Delta^{le_i} u  \eta^k \right|^2  (\tau)+  \int_{t_1}^{t_2} \int_{B_{\rho_2}}  \left| \Delta^{le_i} \V(\nabla u) \eta^k \right|^2 
\le 
\\ \frac{1}{2}  \int_{B_{\rho_2}} \left| \Delta^{le_i} u  \eta^k \right|^2  (t_1)+
  C (G (\vfi')) k
 \int_{t_1}^{t_2} \int_{B_{\rho_2}} 0+ \vfi'_{|\nabla u| } \left(|\Delta^{le_i} \nabla u| \right) | \Delta^{le_i} u|  \eta^{2k-1}|\nabla\eta|\sigma^{2k} \\
 + \eps_0  \int_{t_1}^{t_2} \int_{\Omega^\eps }  \left|  \left(\Delta^{le_i} u \right)  \eta^k \right|^2 + \frac{1}{4 \eps_0}   \int_{t_1}^{t_2}  \int_{\Omega^\eps }  \left| \Delta^{le_i} f \right|
 =: \\    \frac{1}{2}  \int_{B_{\rho_2}} \left| \Delta^{le_i} u  \eta^k \right|^2  (t_1) + C (G (\vfi'))  \int_{t_1}^{t_2} \int_{B_{\rho_2}}  ( 0 + B + \eps_0 C + F).
\end{multline*}
This implies the following analogue of \eqref{eq:jb_strong:spatial:ini:finiX}
\begin{multline*}
\frac{1}{2}  \int_{B_{\rho_2}} \left| \Delta^{le_i} u  \right|^2  (\tau)\eta^{2k}+   \int_{t_1}^{t_2} \int_{B_{\rho_2}}  \left| \Delta^{le_i} \V(\nabla u) \eta^k \right|^2  \le \\ \frac{1}{2}  \int_{B_{\rho_2}} \left| \Delta^{le_i} u \right|^2  (t_1)\eta^{2k}+
 4 \delta  \int_{t_1}^{t_2} \int_{B_{\rho_2}}  |\Delta^{l e_i}  \V ( \nabla u)  \eta^k |^2 +  6 \delta  \int_{t_1}^{t_2}   \dashint_0^l \frac{|l|}{|h|}   \int_{B_{\rho_2}}   |\Delta^{\lambda e_i}  \V ( \nabla u)  \;  \eta^k  |^2 d \lambda  \\
   +  
C (\delta, G(\vfi'))  \frac{|h|^2}{({\rho_2}-{\rho_1})^2}   \dashint_0^l  \int_{Q_{\rho_2}} \vfi \left(  | \nabla u \circ T_{\lambda e_i}  | \right)  d \lambda + \\
\eps_0  \int_{t_1}^{t_2} \int_{B_{\rho_2} }  \left|  \left(\Delta^{le_i} u \right)  \eta^k \right|^2 + \frac{ |h|^2}{4 \eps_0}  \int_{t_1}^{t_2} \int_{B_{\rho_2} }  \dashint_0^l  | \nabla f \circ T_{\lambda e_i} |^2 d \lambda,
\end{multline*}
which implies \eqref{eq:jb_strong:2+_nabla'X} along the respective lines of the proof of Theorem \ref{th1} (after renaming $\rho_2$ to $R_0$ and taking $\eta$ cutting off between $B_R$ and $B_{R_0}$). Theorem \ref{th2} is proved.
\qed

\section{Concluding remarks}\label{sec:crems:space}
As already mentioned at the beginning of the proof, thanks to the use of the equality \eqref{eq:strong:spatial:eff} and uniformization in $h$ (see the step 3 of the above proof), we avoided resorting to an additional covering argument as in \cite{DieEtt08}.

The estimate  \eqref{eq:jb_strong:2+_nabla}, among other advantages, involves explicit dependance on $G(\vfi')$. Hence it is uniform for growth functions that have a common  $G(\vfi')$, for instance for approximative quadratic potential $\vfi_{\lambda} $, presented in subsection \ref{subs:disc}. This in turn will prove very useful for the planned interior two-dimensional smoothness result.

\section{Appendix - Orlicz growths}\label{jb_structure_orlicz}
Here we gather some details on Orlicz structure. In principle, these are known facts, put together for reader's convenience. The exceptions may be\footnote{In the sense that we could not find the references, so it seems that these facts were never stated, but seem quite straightforward}: Proposition \ref{prop:jb_structure:'giveD} (saying that the property $\vfi' (t) \sim t \vfi''(t)$ implies $\delta_2$ for $\vfi$ and $\vfi^*$) and our version of Korn's inequality (Lemma \cite{lem:korn})  By the end of Appendix, we recall the standard examples for Orlicz growths that are admissible in our main results.
 The standard source for theory of Orlicz spaces are monographs of Musielak \cite{Mus83} and of Rao and Ren \cite{RaoRen91}, \cite{RaoRen02}. For our purposes it is more straightforward to follow monograph \cite{MNRR} by M\'alek, Ne\v{c}as, Rokyta \& R\r{u}\v{z}i\v{c}ka and paper of  Diening \& Ettwein \cite{DieEtt08}, because they are PDEs-oriented.  We begin with
 \subsection{$\eN$-functions.}\label{ssec:jb_structure_orlicz:N} The following definition agrees with that of a Young function in Section 1.2.5 of \cite{MNRR} or with Definition 2.1, \cite{DieKap13}.
\begin{mydef}\label{def:app_orlicz:nf}
A real function $\vfi: \er_+ \to \er_+$ is an \emph{$\eN$-function} iff there exists $\vfi'\!: \er_+ \to \er_+$ 
\begin{itemize}
\item that is right-continuous, non-decreasing\footnote{We use here the nomenclature >>non-decreasing<< and >>increasing<<.},
\item that satisfies $\vfi' (0) =0$,  $ \vfi' (t) >0$ for $t>0$ and $\vfi' (+ \infty^-) =+ \infty$ 
\end{itemize}
such that
\begin{equation}\label{def:eq:N}
\vfi (t) = \int_0^t \vfi' (s) \, ds 
\end{equation}
\end{mydef}
Formula \eqref{def:eq:N} is admissible, because  $\vfi'$ is nonnegative and measurable (as monotonous), thus Lebesgue-integrable. On the level on the above definition, $\vfi'$ is just a name for a function.  However, any $\eN$-function $\vfi$, being defined as an integral, is continuous (absolutely continuous on an interval) and a.e. differentiable, with its derivative a.e. equal to $\vfi'$. Moreover, $\vfi$ is convex, because $\vfi'$ is non-decreasing. In what follows, we will assume in fact more smoothness of admissible $\eN$-functions.\\ 
The concept of the $\eN$-function\footnote{Some claim that $\eN$ stands for >>nice<<.} generalizes the power function $\vfi (t) = \frac{1}{p} t^p$. Now, let us generalize its H\"older-conjugate $ \frac{1}{p'} t^{p'}, \; p' = \frac{p}{p-1}$. To this end, for a non-decreasing real function $g$ let us denote with $g^{-1}$ its generalized right-continuous inverse, given by $ g^{-1} (t):= \sup \{ {s \in \er_+}|\; g(s)  \le t \}$.  Now we  introduce
\begin{mydef}\label{def:app_orlicz:*} A \emph{complementary function} to an  $\eN$-function $\vfi $ is 
\[
\vfi^* (t) := \int_0^t (\vfi')^{-1} (s) \, ds 
\]
\end{mydef}
A complementary function is called also the conjugate one. It is also a $\eN$-function, in view of the above formula. It can be alternatively defined without using the generalized inverse by $\vfi^* (t) := \sup_{s \ge 0} (ts - \vfi (s)) $. Then, however, one does  not see immediately that $\vfi^*$ is a $\eN$-function. The equivalence of these two definitions is contained in Theorem 13.6 of \cite{Mus83}. Observe that $(\vfi^*)^* = \vfi$.
 \subsection{Regularity assumptions on $\eN$-functions} 
 The assumption widely used for studying regularity for systems with Orlicz growths, see p. 31 of \cite{MNRR}, is given by
\begin{mydef}
$\eN$-function $\vfi$ satisfies \emph{$\Delta_2$  condition} iff there exists a numerical constant $C$ for which
\begin{equation}\label{eq:jb_weak:delta2}
\vfi (t)  \sim \vfi (2 t)
\end{equation}
$t$-uniformly
In such case we denote the optimal numerical constant of \eqref{eq:jb_weak:delta2} by $\Delta_2 (\vfi)$. $\Delta_2 (\Phi)$ stands for the supremum of such constants over a family $\Phi$ of $\eN$-functions.
\end{mydef}
Any $\eN$-function is increasing, so in fact $ \Delta_2$  condition is equivalent to $ \vfi (2t)  \le \Delta_2 (\vfi) \, \vfi (t)$.
Let us state the basic connection of an $\eN$-function and its conjugate, where  $\Delta_2$  condition is used. It can be found as formulas (2.1), (2.3) of \cite{DieEtt08}

\begin{lem}[good $\vfi$ growth and Young's inequality] \label{lem:jb_weak:goodphi}
Assume that a family $\Psi$ of $\eN$-functions and their conjugates  $(\Psi = \Phi \cup  \Phi^* )$ satisfies common $\Delta_2$ condition. Then it holds  for any $\varphi \in \Psi$ and any $t, a, b \in \er_+$
\begin{equation}\label{eq:jb_weak:goodphi}
t \varphi' (t) \sim \varphi (t)
\end{equation}
\begin{equation}\label{eq:jb_weak:phistar}
\varphi^*( \varphi' (t) ) \sim \varphi (t)
\end{equation}
where all the constants depend only on the magnitude of $\Delta_2 (\Psi)$, thus are uniform for the family $\Psi$.\\
Moreover for any $\delta >0$
\begin{equation}\label{eq:jb_weak:Y}
ab \le \delta \vfi (a) + C(\delta) \vfi^* (b)
\end{equation}
where $C(\delta) $  depends only on $\delta$ an the magnitude of $\Delta_2 (\Psi)$, thus is uniform for the family $\Psi$.
\end{lem}

It is common to further restrict the set of admissible $\eN$-functions in order to mimic even closer the $p$-growth. This is done by imposing the good $\vfi'$ property, compare Definition \ref{def:jb_weak:goodphi'}.

 The $\Delta_2$ condition for $\vfi$ does not imply the  $\Delta_2$ for $\vfi^*$. So it is a little surprise that one has
\begin{prop}[Good $\vfi'$ property implies other properties]\label{prop:jb_structure:'giveD}
Take $\eN$-function $\vfi$ that has good $\vfi'$ property. Then $\vfi$ satisfies also $\Delta_2$ condition and it holds $\Delta_2 (\vfi) \le C \left(G (\vfi') \right)$. Moreover, the complementary function $\vfi^*$ also enjoys good $(\vfi^*)'$ property and $G ((\vfi^*)') \le C( G (\vfi'))$. Hence also $\Delta_2 (\vfi^*) \le C \left(G (\vfi') \right)$.
\end{prop}
\begin{proof}
First we show that $\Delta_2 (\vfi) \le C \left(G (\vfi') \right)$. Formula \eqref{eq:jb_weak:goodphi'} and Definition \ref{def:app_orlicz:nf} give for $s> 0$
\[
\frac{ \varphi'' (s) }{\varphi' (s)} \le \left(G (\vfi')s \right) ^{-1}
\]
so we have after integration over $(\tau, 2\tau)$
\[
\ln \left( \varphi' (2\tau) \right) \le \ln \left( \varphi' (\tau) \right) + \ln 2^\frac{1}{G (\vfi') } = \ln \left( \varphi' (\tau) 2^\frac{1}{G (\vfi') } \right).
\]
This by monotonicity of $\ln$ gives
\[
\varphi' (2 \tau) \le  \varphi' ( \tau) 2^\frac{1}{G (\vfi') } 
\]
which after integration over $(0, t)$ gives  the non-trivial part of $ \Delta_2 $ for $\vfi$. Next we focus on the complementary function $\vfi^*$. It is an  $\eN$-function and by its definition, see Definition \ref{def:app_orlicz:*},  $(\vfi^*(t))' =  (\vfi')^{-1} (t) $. The good $\vfi'$ property formula \eqref{eq:jb_weak:goodphi'} shows that $\vfi''$ is positive for positive arguments. This, together with fact that $(\vfi')^{-1} (t) =0$ only for $t=0$, gives for $t > 0$
\[
(\vfi^*(t))'' =  \left( (\vfi')^{-1} (t)\right)' =   \frac{1}{\vfi'' \left( (\vfi')^{-1} (t) \right)} :=I
\]
Next we use good $\vfi'$ property at $s =   (\vfi')^{-1} (t)$ to get
\[
I \sim    \frac{ (\vfi')^{-1} (t)}{\vfi' \left(  (\vfi')^{-1} (t) \right)} =     \frac{ (\vfi^*(t))' }{t} 
\]
where for the equality we again use  $(\vfi^*(t))' =  (\vfi')^{-1} (t) $. We put together both expressions above that contain $I$ and obtain good $(\vfi^*)'$ property and $G ((\vfi^*)') \le C( G (\vfi'))$. Consequently, we can use first part of this proposition for $\vfi^*$  to write  $\Delta_2 (\vfi^*) \le C \left(G ((\vfi^*)')\right) \le C( G (\vfi')) $. The second inequality follows from the last-but-one sentence.
\end{proof}
\begin{rem}
Proposition \ref{prop:jb_structure:'giveD} implies that in Lemma \ref{lem:jb_weak:goodphi}, used for $\eN$-functions with good $\vfi'$ property, we have 
\[\Delta_2 (\Psi) \le G (\{ {\varphi'} | \; \varphi \in \Phi \})\] 
Hence we control in  Lemma \ref{lem:jb_weak:goodphi} the constants that depend on $\Delta_2 (\Psi)$, which involves $\eN$-function and their conjugates, with $C \left( G (\{ {\varphi'} | \; \varphi \in \Phi \}) \right)$, where the conjugates are not present.
\end{rem}
We will use the above remark as an inherent part of Lemma \ref{lem:jb_weak:goodphi}, when we deal with $\eN$-functions with the good $\vfi'$ property, without referring to it directly.\\
Having assumed good $\vfi'$ property, we are in fact not far away from the $p$-case. More precisely it holds (compare (2.3), (2.4) in \cite{DieKapSch12}) 
\begin{prop}[Boyd indices]\label{prop:jb_structure:polyGrowth}
For an $\eN$-function $\vfi$ that has good $\vfi'$ property, there exist such $q_1 \le  q_2$ from $(1, \infty)$ that for any $t, a \in \er_+$
\[{C^{-1}(G (\vfi'))} (a^{q_1} \! \wedge a^{q_2}) \, \vfi (t) \le  \vfi (at) \le C (G (\vfi')) (a^{q_1} \! \vee a^{q_2}) \, \vfi (t)\]
\end{prop}
Numbers $q_1, q_2$ are sometimes referred to as the Boyd indices. For more on the proximity to the polynomial case and for examples of $\vfi$ functions that enjoy the good $\vfi'$ property and are not of $p$-type, see subsection \ref{subs:disc}.
\subsection{ >>Square root<< of an $\eN$-function and tensor  $\V$.}
In order to gain quadratic structure in estimates we introduce the following >>square root<< of an $\eN$-function $\vfi$ and the associated tensor $\V : Sym^{d \times d} \to Sym^{d \times d}$.

\begin{mydef}[$\Vfi$ and $\V$] \label{def:jb_weak:V}
Let $\vfi$ be an $\eN$-function. We define 
\[\Vfi' := \sqrt{ t \vfi' (t)}, \qquad \V := \partial_Q \Vfi ( |Q|)\]
\end{mydef}
It holds
\begin{lem}\label{lem:jb_weak_sqrtvfi}
For an $\eN$-function $\vfi$, its >>square root<< $\Vfi$ is also an $\eN$-function. $\Vfi$ inherits $\Delta_2$ condition and good $\vfi' $ property from $\vfi$, with $\Delta_2( \Vfi) \le C \left(\Delta_2( \vfi) \right), \, G( \Vfi') \le C  \left( G ( \vfi') \right)$. Moreover, if  $\vfi$ enjoys the good $\vfi'$ property, then uniformly in $t > 0$ holds
\[
\Vfi'' (t) \sim  \sqrt{ \vfi'' (t)} 
\]
with constants in $\sim$ depending only on $G ( \vfi') $.
\end{lem}
It follows from a computation, one can consult also Lemma 25 of  \cite{DieEtt08}.
\subsection{Connection between tensor $\A$ and $\eN$-function $\vfi$} Assumption \ref{ass:jb_weak:growth_gen} provides connection between the thoroughly described Orlicz-growth functions and the tensor $\A$. \\
Observe that for  $P : \er \to  Sym^{d \times d}$ and $\A \circ P : \er \to  Sym^{d \times d}$  differentiable at $s$, Assumption \ref{ass:jb_weak:growth_gen} gives
\begin{equation}\label{eq:jb_structure:diffa1}
\begin{aligned}
\partial_s (\A(P)) : \partial_s P  & \ge  c \vfi'' (|P|) |\partial_s P|^2 \\
| \partial_s (\A(P)) |& \le C \vfi'' (|P|) |\partial_s P|
\end{aligned}
\end{equation}
How does one find for a certain tensor $\A$ the appropriate $\vfi$ of Assumption \ref{ass:jb_weak:growth_gen}? The easiest way is to use the potential of $\A$, if it exists. Therefore it is common to consider the following 
\begin{ass}[Orlicz growth of  $\A$ - potential case]\label{ass:jb_weak:growth_pot}
For a certain  $\eN$-function $\vfi$ that enjoys the good $\vfi' $ property, tensor $\A$ is given by the formula
\begin{equation*}
\A(Q) : = \partial_Q \vfi ( |Q|),
\end{equation*}
where $Q \in Sym^{d \times d}$.
\end{ass}
We have in view of Lemma 21 of \cite{DieEtt08}
\begin{prop}\label{prop:jb_structure:mon_potT}
For an $\eN$-function $\vfi$ that enjoys the good $\vfi' $ property, the tensor $\T$  given by \[\T (Q) :=  \partial_Q \vfi ( |Q|), \qquad Q \in Sym^{d \times d}\] satisfies
$ \T (0) = 0 $ and otherwise  $\T (Q) = \vfi' ( |Q|) \frac{Q}{|Q|}$. Moreover
\begin{equation*}
\begin{aligned}
(\T(P) - \T(Q) ) : (P- Q) & \ge  c(G (\vfi')) \vfi'' (|P| + |Q|) |P-Q|^2 \\
| \T(P) - \T(Q)|& \le C(G (\vfi')) \vfi'' (|P| + |Q|) |P-Q|
\end{aligned}
\end{equation*}
for any $P, Q \in Sym^{d \times d}$.
\end{prop}

Taking in the above proposition $\T := \A$, we have
\begin{cor}\label{prop:jb_structure:mon_pot}
Assumption \ref{ass:jb_weak:growth_pot} implies  Assumption \ref{ass:jb_weak:growth_gen}.
\end{cor}

We have stated Proposition \ref{prop:jb_structure:mon_potT} for an unspecified $\T$ instead of $\A$ because we will need also the following observation. Even if merely Assumption \ref{ass:jb_weak:growth_gen} is valid (and not Assumption \ref{ass:jb_weak:growth_pot}), the tensor $\V$ of Definition \ref{def:jb_weak:V} is itself derived from the potential $\Vfi$. Therefore we can use Proposition \ref{prop:jb_structure:mon_potT} to formulate
\begin{cor}\label{prop:jb_structure:mon_potV}
Assume that an $\eN$-function $\vfi$ has the good $\vfi' $ property. Then for the >>square root<< tensor $\V$ of Definition \ref{def:jb_weak:V} holds
\begin{equation*}
\begin{aligned}
(\V(P) - \V(Q) ) \!: \! (P- Q) & \ge  c\, (G (\vfi')) \, \Vfi'' (|P| + |Q|) |P-Q|^2 \\
| \V(P) - \V(Q)|& \le C(G (\vfi')) \Vfi'' (|P| + |Q|) |P-Q|
\end{aligned}
\end{equation*}
for any $P, Q \in Sym^{d \times d}$.
\end{cor}
\begin{proof}
Lemma \ref{lem:jb_weak_sqrtvfi} implies that for an $\eN$-function $\vfi$ that has the good $\vfi' $ property, the potential $\Vfi$ for $\V$ is also an $\eN$-function enjoying the good $\Vfi' $ property, and that $G (\Vfi') \le C(G (\vfi'))$. Therefore we can use Proposition \ref{prop:jb_structure:mon_potT} with $\T := \V$.
\end{proof}

\subsection{Expressing flexibly the monotonicity. Shifted $\eN$-functions}
The results of this subsection gives a few ways of encoding the monotonicity conditions for $\A$ and show connection between $\A$ and $\V$. To state it, we first present the following ingenious concept of a shifted $\eN$-function. It was supposedly stated in an explicit form first by Diening and Ettwein in \cite{DieEtt08}.
\begin{mydef}\label{def:jb_weak:shifted}
For an $\eN$-function $\vfi$ and $a \ge 0$ the 'shifted $\eN$-function' is given as
\begin{equation}
\vfi_a (t) := \int_0^t \vfi' (a+s ) \frac{s}{a+s} ds
\end{equation}
\end{mydef}
The following lemma gathers important features of
shifted $\eN$-functions.
\begin{lem}\label{lem:app:shifted}
The shifted $\eN$-function $\vfi_a$ is indeed an $\eN$-function. It inherits from $\vfi$  the $\Delta_2 $ condition uniformly with respect to its shift, i.e. for a numerical constant C holds
\begin{equation*}
\Delta_2 ( \{ \vfi_a | a \ge 0 \} ) \le C \Delta_2 (  \vfi  )
\end{equation*}
The same holds for the good $\vfi'$ property with 
\begin{equation*}
\sup_{a \in \er_+} G  ( \vfi'_a  ) \le C G  ( \vfi' )
\end{equation*}
Moreover, if $\vfi$ has the good $\vfi'$ property, then $\vfi''_a (t) \sim \vfi'' (a+t)$ and uniformly in $S, \, T \in Sym^{d \times d}$ one has for $|S| + |T| > 0$ 
\begin{equation}\label{eq:app:shifted:eq}
\begin{aligned}
\vfi'' (|S| + |T|)  |S-T| & \sim   \vfi'_{|S|} (|S - T|) \\
\vfi'' (|S| + |T|) |S-T|^2 & \sim   \vfi_{|S|} (|S - T|) 
\end{aligned}
\end{equation}
and constants in $\sim$ depend only on $G (  \vfi'  )$. One can also change shift as follows
\begin{equation}\label{eq:app:shifted:change}
 \vfi'_{|S|} (|S - T|) \le C (G(\vfi')) \left( \vfi'_{|C|} (|S - C|) + \vfi'_{|C|} (|T - C|) \right)
\end{equation}
where  $C$ is chosen arbitrarily from  $ Sym^{d \times d}$. Moreover, for $\lambda \in [0,1]$ it holds
\begin{equation}\label{eq:app:shifted:little2}
 \vfi_{|S|} (\lambda |S|) \le \lambda^2 C (G(\vfi')) \vfi(|S|)
\end{equation}
\end{lem}\begin{proof}
The fact that $\vfi_a$ is an $\eN$-function and the uniform estimate for $\Delta_2$ follow from the proof of Lemma 23, \cite{DieEtt08}. To show its good behavior with respect to the good $\vfi'$ property take $t>0$ and compute
\[
\vfi''_a (t) = \vfi'' (a+t ) \frac{t}{a+t} +  \vfi' (a+t ) \frac{a}{(a+t)^2} \sim  \vfi' (a+t ) \frac{t}{(a+t)^2} +  \vfi' (a+t ) \frac{a}{(a+t)^2}  = \vfi' (a+t ) \frac{1}{a+t}
\]
Multiplying the above equivalence with $t$ and using Definition \ref{def:jb_weak:shifted} we obtain good $\vfi'$ property for $\vfi_a$ with $G  ( \vfi'_a  ) \le C G  ( \vfi' )$ (in fact, with  $G  ( \vfi'_a  ) = (1 \vee G  ( \vfi' )) =  G  ( \vfi' )$, but we keep $C$, as it anyway appears in the uniform estimate for $\Delta_2$). The formulas \eqref{eq:app:shifted:eq}, \eqref{eq:app:shifted:change}, \eqref{eq:app:shifted:little2}  are given by the proofs of Lemma 24 of \cite{DieEtt08}, Lemma 29 and Lemma 30 therein, respectively.
\end{proof}
Thanks to the above result we have
\begin{cor}\label{cor:jb_weak:goodphi_shifted}
Lemma \ref{lem:jb_weak:goodphi} holds  for  $\{ \vfi_a | a \ge 0 \}$ with the constants majorized $a$-uniformly  by $ C ( G (\vfi')) $.
\end{cor}
We will need also 
\begin{prop}\label{prop:jb_weak:shiftedO}
For an $\eN$-function $\vfi$ with the good $\vfi'$ property and $\vfi''$ almost increasing, we have the implication \[a \le b \implies \vfi_a (t) \le C (G (\vfi'))  \vfi_b (t)\]
\end{prop}
\begin{proof}
The Definition \ref{def:jb_weak:shifted} of a shifted $\eN$-function, the good $\vfi'$ property and the fact that $\vfi''$ is almost increasing give for  $ a \le b$
  \begin{multline*}
\vfi_a (t) = \int_0^t \frac{ \vfi' (a+s )}{a+s} s ds \le C (G (\vfi')) \int_0^t \vfi'' (a+s ) s ds \le \\
C (G (\vfi')) \int_0^t \vfi'' (b+s ) s ds  \le C (G (\vfi')) \int_0^t \frac{ \vfi' (b+s )}{b+s} s ds = \vfi_b (t)
  \end{multline*}

\end{proof}

The following results show how one can interchangeably use $\A, \V $ and $\vfi_a $. 
\begin{prop}[Equivalent notions of monotonicity] \label{lem:jb_weak:mon_equiv}
Assume that an $\eN$-function $\vfi$ enjoys the good $\vfi'$ property. Then for the tensor $\V$ given by the Definition \ref{def:jb_weak:V} holds for any $P, Q
 \in Sym^{d \times d}$ \begin{equation}\label{eq:lem:jb_weak:mon_equiv1}
 \vfi'' (|P| + |Q|) |P-Q|^2   \sim  \vfi_{|P|} ( |P-Q|)  \sim | \V(P) - \V(Q) |^2
\end{equation}
Moreover, if the tensor $\A$ satisfies Assumption \ref{ass:jb_weak:growth_gen} with our  $\vfi$,  the above expressions  are comparable to the monotonicity formula, i.e.
\begin{equation}\label{eq:lem:jb_weak:mon_equiv2}
(\A(P) - \A(Q) ) \!:\! (P- Q)  \sim  \vfi'' (|P| + |Q|) |P-Q|^2   \sim  \vfi_{|P|} ( |P-Q|)  \sim | \V(P) - \V(Q) |^2
\end{equation}
for any $P, Q
 \in \er^{d \times d}$. All constants depend only on $ G (\vfi')$.
\end{prop}
\begin{proof}
The left $\sim$ of \eqref{eq:lem:jb_weak:mon_equiv1} comes from \eqref{eq:app:shifted:eq} of Lemma \ref{lem:app:shifted}. The right $\sim$ of \eqref{eq:lem:jb_weak:mon_equiv1} follows from Corollary \ref{prop:jb_structure:mon_potV} and formula $\Vfi'' (t) \sim  \sqrt{ \vfi'' (t)} $ of Lemma \ref{lem:jb_weak_sqrtvfi}. Now \eqref{eq:lem:jb_weak:mon_equiv2} follows from Assumption \ref{ass:jb_weak:growth_gen}.
\end{proof}
From the Proposition \ref{lem:jb_weak:mon_equiv} we see that for $P : \er \to  Sym^{d \times d}$ and $\V \circ P : \er \to  Sym^{d \times d}$  differentiable at $s$ we have
\begin{equation}\label{eq:jb_structure:diffa15}
 \vfi'' (|P| ) |\partial_s P|^2  \sim  \left|\partial_s\! \left(\V (P)\right) \right|^2
\end{equation}

This gives for  $\A \circ P : \er \to  Sym^{d \times d}$  differentiable at $s$
\begin{equation}\label{eq:jb_structure:diffa2}
| \partial_s (\A(P)) | \le  C (G (\vfi'))  \vfi'' (|P|) |\partial_s P| \le  C (G (\vfi'))  \left( |\partial_s \V (P)|^2 + \vfi'' (|P|) \right) 
\end{equation}
where we use in the first estimate inequality \eqref{eq:jb_structure:diffa1}.
This in turn gives
\begin{equation}\label{eq:jb_structure:diffa2'}
| \partial_s (\A(P)) | |\partial_s P| \le C \vfi'' (|P|) |\partial_s P|^2  \le   C (G (\vfi'))  \left|\partial_s \left(\V (P)\right) \right|^2
\end{equation}

Moreover, for an $\eN$-function $\vfi$ as in Definition \ref{def:jb_weak:goodphi'}, we see from Proposition \ref{lem:jb_weak:mon_equiv} that 
\begin{equation}\label{eq:jb_structure:diffa3}
 |\partial_s P|^2 \le \frac{C (G (\vfi'))}{\vfi'' (0)}  | \partial_s (\V(P)) |^2
\end{equation}
because, by definition, $\vfi$ has second derivatives almost increasing and $\vfi'' (0) > 0$.

\subsection{Korn's inequality} \label{subs:Korn}
\begin{lem}\label{lem:korn}
Let $f \in W^{1,\vfi}_x (\O)$, where $\vfi$ has good $\vfi'$ property. Then there exists $C (\Delta_2 (\vfi, \vfi^*))$ such that
\begin{equation}\label{eq:korn}
\int_{B_r} \vfi (|\nabla u|) \le C (G (\vfi')) \int_{B_r} \vfi (|\D u|) +  \vfi \left( \frac{|u - (u)|}{r} \right)
\end{equation}
\end{lem}
For Orlicz and weighted $L^p$ spaces, the standard source of Korn's inequality is paper by Diening, R\r u\v zi\v cka \& Schumacher \cite{DieRuzSch10}. It contains the needed by us inequality \eqref{eq:korn} in the case of weighted $L^p$ spaces (see formula (5.19) in Theorem 5.17 there), but lacks its Orlicz version (see Theorem 6.13 there). In \cite{DieRuzSch10}, the step from weighted $L^p$ spaces to Orlicz growths follows from the (miraculous) extrapolation technique that originates from \cite{RdF84} by Rub\'io de Francia. In the Orlicz context it says, roughly speaking, that if, for a certain family $\FF$ of pairs $(f;g) \in \FF$,  an inequality holds  in a weighted $L^p$ space for any weight in a Muckenhoupt class $A_p$, then it holds for  $(f;g) \in \FF$ also for (sufficiently regular) Orlicz growths\footnote{Since \emph{there are no Banach function spaces, only weighted $L^2$}, compare Cruz-Uribe, Martell \& Perez \cite{C-UMarPer11}, p. 14}, see \cite{DieRuzSch10}, Proposition 6.1. Hence, the extrapolation-based step from weighted-$L^p$ formula (5.19) in Theorem 5.17 to its Orlicz counterpart would basically consist in choosing  $\FF$ appropriately. Below, we present a proof that does not involve extrapolation.
\begin{proof}
Using monotonicity, convexity and $\Delta_2$ condition, we have
\begin{equation}\label{eq:korn1}
\int_{B_r} \vfi (|\nabla u|) \le C (\Delta_2 (\vfi)) \int_{B_r} \vfi (|\nabla u - (\nabla u)|) ) + \vfi (| (\nabla u)|)  \le C (\Delta_2 (\vfi, \vfi^*)) \int_{B_r} ( \vfi (|\D u - (\D u)|)   + \vfi (| (\nabla u)|), 
\end{equation}
where the second inequality follows from Korn's inequality for oscillations, see Theorem 6.13 of  \cite{DieRuzSch10}. Now it suffices to deal with $ \vfi (| (\nabla u)|)$. To this end let us define $h$ as
\[
h (t) := \vfi (t^\frac{1}{q}).
\]
We need two facts on $h$. 
The first  one is that we can find such $q_0 > 1$ that $h$ is convex, because for $s = t^\frac{1}{q}$
\[
h'' (t) = q^{-2}  \vfi'' (s)  s^{2-2q} + q^{-1}  (q^{-1} -1)  \vfi' (s)  s^{1-2q} \ge \vfi'' (s)  s^{2-2q} \left[  q^{-2}   + C (G (\vfi'))  q^{-1}  (q^{-1} -1)  \right]
\]
and the term in the square brackets above can be made positive by choosing $q = q_0$ sufficiently close to $1$.
The second one is that 
\begin{equation}\label{eq:korn2}
t = \vfi ((h^{-1} (t))^\frac{1}{q} )
\end{equation}
since $s = \vfi ( \vfi^{-1} (s)) = h ( (\vfi^{-1})^q (s))$.
We are ready to deal with $ \vfi (| (\nabla u)|)$ on r.h.s. of \eqref{eq:korn1}
\[
 \vfi (| (\nabla u)|) \le  \vfi \left( \left( \dashint_{B_r} |\nabla u|^{q_0} \right)^\frac{1}{q_0} \right) \le   \vfi  \left( C \left( \dashint_{B_r} \left(  |\D u|^{q_0} +   \frac{|u - (u)|^{q_0}}{r^{q_0}} \right) \right)^\frac{1}{q_0}  \right) =: I
\]
The first inequality above follows from Jensen's inequality and monotonicity of $\vfi$, the second one from the Korn's inequality for Lebesgue spaces (see for instance formula (5.19) in Theorem 5.17 with $w \equiv 1$). Next, we plug $h^{-1} \circ h \equiv 1$ into $I$ and use convexity of $h$ to get
\begin{multline*}
I =   \vfi  \left( C \left(h^{-1} \circ h \left(  \dashint_{B_r} \left(  |\D u|^{q_0} +   \frac{|u - (u)|^{q_0}}{r^{q_0}} \right) \right) \right)^\frac{1}{q_0}  \right)  \le  C \vfi  \left(  \left(h^{-1} \left(  \dashint_{B_r} h \left(   |\D u|^{q_0} +   \frac{|u - (u)|^{q_0}}{r^{q_0}} \right) \right) \right)^\frac{1}{q_0}  \right) \\ 
= C \dashint_{B_r} h \left(   |\D u|^{q_0} +   \frac{|u - (u)|^{q_0}}{r^{q_0}} \right) \le C \dashint_{B_r} h \left(   |\D u|^{q_0} \right)  +h \left(     \frac{|u - (u)|^{q_0}}{r^{q_0}} \right),
\end{multline*}
the second equality is given by \eqref{eq:korn2} and the last inequality follows from convexity and validity of $\Delta_2$ condition for $h$ (see its definition). Putting together the estimates involving $I$ we have via definition of $h$
\[
 \vfi (| (\nabla u)|) \le C \dashint_{B_r} h \left(   |\D u|^{q_0} \right)  +h \left(     \frac{|u - (u)|^{q_0}}{r^{q_0}} \right) = C \dashint_{B_r} \vfi \left(   |\D u| \right)  + \vfi \left(     \frac{|u - (u)|}{r} \right),
\]
with $C$ depending only on $G (\vfi')$. The above estimate used in  \eqref{eq:korn1} implies thesis.
\end{proof}

\subsection{Examples of admissible growths} \label{subs:disc}
\[
 \A^1 (Q) := (\mu + |Q|^{p-2} ) Q,  \qquad  \A^2 (Q) := (\mu + |Q|^2)^\frac{p-2}{2} Q
\]
with $\mu \ge 0$, $p > 1$,  provide us with the (symmetric) $p$-Laplacian prototypes. Tensors $\A^1, \A^2$ are given by the following $p$-potentials 
\[
\vfi^1 (t) = \int_0^t \! (\mu + s^{p-2}) \, s \,ds,  \qquad  \vfi^2 (t) = \int_0^t \!(\mu + s^2)^\frac{p-2}{2} \, s \,ds
\] 
The respective square root tensors $\V$ read
\[ 
\V^1 (Q) = \sqrt{\mu + |Q|^{p-2}} Q , \quad   \V^2 (Q) = (\mu + |Q|^2)^\frac{p-2}{4} Q
\]
An example of an  admissible potential connected with  a non-polynomial growth reads
\[
   \vfi^3 (t) = \int_0^t (\mu +s)^{p-2} s \ln (e+s) ds
   \]
where the  associated tensors are given by
\[
 \A^3 (Q)= (\mu + |Q|)^{p-2} \ln (e+|Q|) Q ,  \quad \V^3 (Q) = (\mu + |Q|)^\frac{p-2}{2}  \sqrt{ \ln (e+|Q|)} Q.
   \]
Unfortunately, already the $\Delta_2$-condition, widely used in the regularity theory for systems with generalized growths, excludes exponential or $L\log{L}$ growths, that are of some interest from the perspective of applications.

\subsection*{Acknowledgement}
JB is partially supported by the National Science Centre (NCN) grant no. 2011/01/N/ST1/05411.



\bibliographystyle{abbrv}

\end{document}